\documentclass[10pt,reqno]{amsart}
\usepackage[cp1251]{inputenc}
\usepackage[english]{babel}
\usepackage{amsmath}
\usepackage{amssymb}
\usepackage{amsfonts}
\usepackage{graphicx}
\usepackage{eucal}
\usepackage{hyperref,amsthm}
\usepackage{xcolor}
\RequirePackage{bbold}
\usepackage{mathptmx,bm}
\usepackage{bbm}

\newtheorem{theorem}{Theorem}
\newtheorem{lemma}{Lemma}

\newtheorem{corollary}{Corollary}
\theoremstyle{definition}

\newtheorem{remark}{Remark}
\numberwithin{equation}{section}
\usepackage{tikz}
\usepackage{pgfplots}

\begin{document}

\title[Second--order renewal asymptotics in the finite--variance regularly varying regime]
    {Second--order renewal asymptotics in the finite--variance regularly varying regime}

\maketitle
\begin{center}
{\large \textbf{{Azam~A.~Imomov}$^{1,*}$ and {Shakhzod~E.~Rizaqulov}$^{2}$}}\\

\vspace{0.5em}

{\small
$^{1}$ Karshi State University, 17, Kuchabag st., 180100 Karshi city, Uzbekistan \\
$^{2}$ Tashkent University of Information Technologies, 100000 Tashkent, Uzbekistan \\
}

\vspace{0.5em}

{\small *Corresponding author: {imomov\_azam@mail.ru}}

\vspace{1em}

\emph{Dedicated to the Fond Memory of our Great Ancestry}
\end{center}

\begin{abstract}
    We develop a refined asymptotic theory for renewal processes with regularly varying inter-arrival
    distributions in the finite-variance regime. Assuming that $1-F(t)=t^{-\alpha}L(t)$ with $\alpha\in(2,3)$
    and $L$ slowly varying, we establish second-order asymptotic expansions for both the renewal convolution
    and the renewal function. The results reveal a two-scale structure: the equilibrium tail $Q_F(t)$ governs
    the local renewal correction, while the integrated tail $\rho(t)=\int_t^\infty Q_F(u)\,du$ determines
    the global deviation from equilibrium. A particular emphasis is placed on the critical threshold $\alpha=2$,
    where the asymptotic behaviour undergoes a structural transition. In this borderline case, the power-law
    hierarchy collapses and the dominant correction is governed by an integrated slowly varying tail. Nevertheless,
    the local renewal correction mechanism persists, and the correction term remains asymptotically proportional
    to the equilibrium tail. The analysis is based on a systematic use of Laplace-transform techniques and
    Tauberian transfer principles, which allow us to relate singular behaviour in the Laplace domain to precise
    time-domain asymptotics. The results provide a unified description of second-order renewal asymptotics
    across the entire range $2\leq\alpha<3$, and point to possible extensions to renewal equations
    and to age-dependent branching processes.

\vspace{0.5em}

\noindent\emph{\bf Keywords}: {Renewal theory; Regularly varying tails; Equilibrium distribution;
    Heavy-tailed distributions; Slowly varying functions; Stone constant; Second-order asymptotics; Tauberian theorems; Laplace transforms; Critical regime; Renewal density; Remainder estimates.}

\vspace{0.5em}

\noindent\textbf{MSC (2020):} {Primary 60K05; Secondary 44A10, 60G70}
\end{abstract}

\maketitle

\section{Introduction: Background, Motivation and Objective}        \label{IRSec1}

    Renewal theory provides a fundamental framework for describing the long-term behavior of systems
    evolving through successive random events. Let ${\mathbb{N}}_0:=\{0\}\cup{\mathbb{N}}$ and
    $\mathbb{R}_+:=[0,\infty)$. The renewal measure $\mathbb{U}$ associated with a measurable set
    $\mathcal{J}\subset\mathbb{R}_+$ is defined as the expected number of renewal epochs falling within
    that set. Given a probability measure $\mathbb{F}$ representing the common distribution of inter-arrival
    times $\tau, \tau_1, \tau_2, \ldots\,$, the renewal measure is the sum of the $n$-fold convolutions of
    $\mathbb{F}$ with itself:
\[
    \mathbb{U}\left(\mathcal{J}\right)=\sum_{n\in{\mathbb{N}_0}}{\mathbb{F}^{\ast{n}}\left(\mathcal{J}\right)},
\]
    where ${\mathbb{F}^{\ast{n}}}$ denotes the $n$-fold convolution of $\mathbb{F}$ and
    ${\mathbb{F}^{\ast{0}}}=\delta_0$ is the Dirac measure (unit mass) at the origin, representing
    the initial event at time zero. The measure $\mathbb{U}$ satisfies the renewal equation
    $\mathbb{U}=\delta_0+\mathbb{U}\ast\mathbb{F}$, which provides the measure-theoretic formulation
    of the renewal identity. For the special choice $\mathcal{J}=[0,t]$, the renewal measure reduces
    to the renewal function $U(t):=\mathbb{U}\left([0,t]\right)$, representing the expected number
    of renewals occurring in the time interval $[0,t]$. Accordingly, $F(t):=\mathbb{F}\left(\tau\in[0,t]\right)$
    denotes the common distribution function of the inter-arrival times.

    The renewal function $U(t)$ plays a central role in renewal theory. In most theoretical and applied
    contexts, qualitative or quantitative information about $U(t)$, and in particular its asymptotic behaviour
    as $t\to\infty$, is sufficient to resolve the principal questions concerning the long-term dynamics of
    the underlying system. The oldest and most fundamental result in this direction is the elementary renewal
    theorem, established in full generality by Feller~\cite{Feller41}, which asserts that
\begin{equation}     \label{AISR1.1}
    \frac{U(t)}{t}\longrightarrow\frac{1}{\,\mu\,}
    \qquad \text{as} \quad t\to\infty,
\end{equation}
    where $\mu=\mathbb{E}\tau\in(0,\infty]$, with the convention that the limit is zero when $\mu=\infty$.

    Historically, the first rigorous proofs of the renewal limit theorem relied on Tauberian arguments for
    Laplace transforms, rooted in the foundational works of Ikehara~\cite{Ikehara31} and Wiener~\cite{Wiener32}.
    A decisive refinement was later obtained by Blackwell~\cite{Blackwell48}, whose celebrated theorem describes
    the asymptotic behaviour of renewal increments in the non-lattice case. These developments were subsequently
    unified and significantly extended by Smith~\cite{Smith54,Smith58} through his formulation of the key
    renewal theorem:
\[
    (g\ast U)(t)\longrightarrow\frac{1}{\,\mu\,} \int_{\mathbb{R}_+}g(x)\,dx
    \qquad \text{as} \quad t\to\infty,
\]
    whenever $g(\cdot)$ is directly Riemann integrable. In particular, the elementary renewal
    theorem~\eqref{AISR1.1} may be viewed as a special case of this result, corresponding, at a heuristic
    level, to the choice $g\equiv{1}$. This observation highlights the fact that \eqref{AISR1.1} represents
    the simplest instance of a much broader convolutional asymptotic principle. Further classical developments
    and a detailed exposition of renewal asymptotics are available in the monographs {\cite[Ch.~X]{Borovkov13}},
    {\cite{Cox62}} and \cite[Ch.~XI]{Feller71}.

    Stone~{\cite{StoneAMS65}},~{\cite{StoneTAMS65}} and later Carlsson~{\cite{Carlsson83}} showed that the renewal
    theorem is governed by the local behavior of the characteristic function near the origin, thereby making
    remainder terms explicitly dependent on moment conditions. In parallel, Chover and Ney~{\cite{ChoverN68}}
    established that the linear asymptotic growth of the renewal function remains robust under broad nonlinear
    extensions. Nevertheless, both frameworks rely critically on the existence of finite moments of order strictly
    greater than two, which ensure sufficient smoothness for higher-order expansions and quantitative error control.
    This observation naturally raises the question of how the renewal structure is altered when such regularity
    conditions fail. We are thus led to consider a heavy-tailed framework of the form
\begin{equation} \label{AISR1.2}
    \overline{F}(t):=1-F(t)=t^{-\alpha}L(t) \qquad \text{for} \quad {0<\alpha<3},
\end{equation}
    where $L(\cdot)$ belongs to $\mathcal{SV}_\infty$, the class of slowly varying functions at infinity.
    Recall that a measurable function $L:\mathbb{R}_+\to\mathbb{R}_+$ is said to be \emph{slowly varying (SV)
    at infinity} (in the sense of Karamata) if $\lim_{t\to\infty}{L(ct)}/{L(t)}=1$ for every fixed
    $c\in\mathbb{R}_+$. This class forms the natural analytic scale for describing asymptotic regimes beyond
    polynomial order. More generally, a function $R:\mathbb{R}_+\to\mathbb{R}_+$ is called \emph{regularly
    varying (RV) at infinity} with index $\rho\in(-\infty,+\infty)$, written $R\in\mathcal{RV}_{\infty}(\rho)$,
    if $\lim_{t\to\infty}{R(ct)}/{R(t)}=c^\rho$. Functions in this class admit the canonical representation
    $R(t)=t^{\rho}L(t)$ with $L\in\mathcal{SV}_\infty$. This decomposition separates the principal
    polynomial scaling from a SV correction and will play a central role in the asymptotic analysis
    throughout the paper; see Bingham, Goldie and Teugels~{\cite{BinghGT87}}. Consequently, {\eqref{AISR1.2}}
    implies that the distribution tail $\overline{F}(t)\in\mathcal{RV}_{\infty}(-\alpha)$. The parameter
    range $0<\alpha<3$ therefore yields three structurally distinct regimes:

\begin{itemize}
\item[\textbf{(i)}] $0<\alpha<1$ \quad  \,\textit{the infinite-mean case}; \label{AISR(i)}
\vspace{1.5mm}
\item[\textbf{(ii)}] $1<\alpha<2$ \quad  \,\textit{the finite-mean but infinite-variance regime}; \label{AISR(ii)}
\vspace{1.5mm}
\item[\textbf{(iii)}] $2<\alpha<3$ \quad  \,\textit{the finite-variance but infinite third-moment regime}. \label{AISR(iii)}
\end{itemize}

    In this framework, renewal asymptotics are no longer organized by integer moments but by the tail index $\alpha$. The
    natural analytical language therefore becomes that of Karamata's regular variation theory. Consequently, Tauberian
    methods and Karamata-type arguments replace classical moment expansions in deriving precise asymptotics
    and remainder estimates.

    The systematic study of regime~{\hyperref[AISR(i)]{\textbf{(i)}}} dates back to the early 1960s. Smith~{\cite{{Smith61}}}
    investigated renewal growth under $\mu=\infty$, demonstrating that the classical linear normalization must be replaced
    by a scale determined by the tail behavior of the distribution. Subsequently, Teugels~{\cite{Teugels68}} developed
    a systematic treatment within the framework of regular variation, establishing a form of the key renewal theorem for
    $0<\alpha\le{1}$. A decisive refinement was obtained by Erickson~{\cite{Erickson70}}, who proved strong renewal theorems
    for the increments $U(t+h)-U(t)$ in the RV setting, thereby identifying the precise asymptotic rate in terms of the
    truncated mean function $\mu(t)=\int_0^t\overline{F}(x)\,dx$. These results firmly established that, in the infinite-mean
    regime, renewal growth is governed entirely by the tail index $\alpha$ rather than by classical moment expansions.

    The regime~{\hyperref[AISR(ii)]{\textbf{(ii)}}} corresponds to the case of finite mean but infinite variance. In this
    regime, the renewal function admits a linear principal term $t/\mu$, yet the structure of the remainder becomes genuinely
    non-classical. Teugels~\cite{Teugels68} showed that if $\overline{F}(t)\in\mathcal{RV}_{\infty}(-\alpha)$ with $1<\alpha<2$,
    then the centered renewal function $U(t)-t/\mu$ belongs to $\mathcal{RV}_{\infty}(2-\alpha)$, thereby revealing that the
    deviation from linear growth is still dictated by the tail index. This perspective was further developed by
    Mohan~\cite{Mohan76}, who connected the renewal remainder to stable domains of attraction, thus providing a probabilistic
    interpretation of the asymptotic structure.  A concise synthesis of these developments, including explicit asymptotic
    formulas for $U(t)-t/\mu$, is provided in the monographic treatment of Mitov and Omey~{\cite{MitovO14}}, where the
    RV remainder is formulated in a unified framework.

    Narrowing our focus to the regime~{\hyperref[AISR(iii)]{\textbf{(iii)}}}, we enter a finite-variance heavy-tailed regime
    in which the classical second-order structure of the renewal theorem still persists, while the third moment diverges.
    In this transitional setting, the principal linear term survives, yet further refinement can no longer rely on higher-order
    moment expansions. Recent results, notably those of Geluk and Frenk~{\cite{GelukF2011}}, indicate that the next-order
    correction is governed by the integrated tail, revealing that even under finite variance, renewal remainders are driven
    by the RV structure rather than by successive integer moments. Thus, this regime forms a delicate boundary zone between
    classical smooth asymptotics and genuinely heavy-tailed behavior.

    Motivated by this boundary regime, we aim to characterize the influence of the RV tail on the rate
    of convergence in the elementary renewal theorem and to derive sharp remainder estimates under
    minimal assumptions. An important companion object is the renewal density $u(t):=U'(t)$, when it
    exists. While the elementary renewal theorem captures the leading behaviour of $U(t)$, refined
    asymptotics for both $U(t)$ and $u(t)$ are often required, yet become delicate under the heavy-tailed
    framework~\eqref{AISR1.2}, where moment-based expansions fail.

    Our primary objective is to develop a second-order asymptotic theory in the finite-variance RV regime,
    quantifying the leading deviations of $U(t)$ and $u(t)$ in terms of the tail index $\alpha$. The approach
    combines Laplace-transform analysis with Tauberian transfer, linking singular behaviour at the origin
    to precise time-domain asymptotics. We obtain explicit second-order asymptotics for $U(t)$ and $u(t)$
    for $2<\alpha<3$, provide a unified Laplace-domain framework, and isolate the underlying Tauberian
    mechanism. The critical case $\alpha=2$ is treated separately, where logarithmic effects emerge.   
    
    Our contribution refines and extends~{\cite{GelukF2011}} and~{\cite{MitovO14}} in three respects. First, 
    where the key renewal theorem of Mitov and Omey~{\cite[Th~1.14]{MitovO14}} identifies the leading renewal 
    remainder and Geluk and Frenk~{\cite{GelukF2011}} the leading finite-variance deviation of $U(t)$ from 
    its linear term $t\big/\mu$ through the integrated tail, {\hyperref[AISRTh1]{Theorem~\ref{AISRTh1}}} and
    {\hyperref[AISRTh2]{Theorem~\ref{AISRTh2}}} resolve that deviation to second order, separating a global 
    correction on the scale of the integrated tail from a local correction on the equilibrium-tail scale itself --- 
    a two-scale	structure not made explicit in either work. Second, neither treats the renewal density: 
    {\hyperref[AISRTh3]{Theorem~\ref{AISRTh3}}} and {\hyperref[AISRTh5]{Theorem~\ref{AISRTh5}}} supply its 
    second-order asymptotics, the local counterpart of the cumulative theory. Third, both remain in the strict 
    finite-variance interior $\alpha>2$, whereas {\hyperref[AISRTh4]{Theorem~\ref{AISRTh4}}}, and 
    {\hyperref[AISRCor2]{Corollary~\ref{AISRCor2}}} reach the critical threshold $\alpha=2$, where 
    the power-law hierarchy collapses and the second-order scale passes to the integrated SV tail 
    $\int_{t}^{\infty}\left[L(u)\big/u\right]\mathbf{d}u$, with logarithmic corrections beyond the reach 
    of the finite-variance expansions of~{\cite{GelukF2011}} and~{\cite{StoneAMS65}}. All three refinements 
    issue from a single Laplace-domain mechanism, developed uniformly across the interior and the boundary 
    in Section~\ref{IRSec2}.

    The remainder of the paper is organized as follows. In Section~\ref{IRSec2}, we present the
    principal results, including second-order asymptotics for the renewal function and the renewal
    density, as well as the critical case $\alpha=2$. Section~\ref{IRSec3} develops the underlying
    Laplace-transform analysis, including the asymptotic structure in the finite-variance regime,
    the critical case, and the Tauberian transfer mechanisms. The proofs of the main results are
    collected in Section~\ref{IRSec4}. Finally, Section~\ref{IRSec5} contains concluding remarks.

\section{Principal Results}      \label{IRSec2}

    In the regime~{\hyperref[AISR(i)]{\textbf{(i)}}}, the leading asymptotics are entirely tail-driven,
    leaving no natural analogue of a second-order refinement around a linear principal term. In the
    regime~{\hyperref[AISR(ii)]{\textbf{(ii)}}}, while a linear principal term $t/\mu$ re-emerges, the
    remainder is governed by stable-type behavior and is itself RV with index $2-\alpha$.
    The deviation from linear growth is still dictated by the heavy-tailed structure, and no finite-variance
    second-order expansion is available. The transition at $\alpha=2$ marks a structural turning point in
    the asymptotic behavior of the renewal function, to which we shall return with special attention below.
    For the regime~{\hyperref[AISR(iii)]{\textbf{(iii)}}}, the variance becomes finite while higher-order
    integer moments remain infinite, thereby creating a delicate boundary regime in which the classical
    renewal structure partially survives while moment-based refinements cease to apply. It is precisely
    this finite-variance RV regime that forms the focus at this stage.

\subsection{Second-order asymptotics for the renewal function}      \label{IRSsec2.1}

    Assume henceforth that $F$ is nonlattice. Introduce the equilibrium (stationary-excess)
    distribution associated with $F$,
\begin{equation}       \label{AISR2.1}
    \mu_F(t):=\frac{1}{\,\mu\,}{\int_{[0,t]}{\overline{F}(u)}\,du},
\end{equation}
    which defines a proper distribution function on $[0,\infty)$. A direct calculation shows that
    this transformation linearizes the renewal structure in the sense that $(\mu_F\ast{U})(t)=t\big/\mu$.
    Passing to the associated equilibrium tail $Q_F(t):=1-\mu_F(t)$, we obtain the canonical
    renewal decomposition
\begin{equation}       \label{AISR2.2}
    ({Q_F}\ast{U})(t)=U(t)-\frac{t}{\,\mu\,}{\raise1.5pt\hbox{,}}
\end{equation}
    which isolates the deviation of the renewal function from its principal linear term.
    As shown in~\cite[Th~1.14]{MitovO14}, the convolution term in~{\eqref{AISR2.2}}
    admits the first-order asymptotic behaviour
\[
    ({Q_F}\ast{U})(t)\sim \frac{1}{\,\mu\,}R(t) \qquad \text{as} \quad t\to\infty,
\]
    where
\begin{equation}       \label{AISR2.3}
    R(t)=\int_{[0,t]}{Q_F}(u)\,du
\end{equation}
    is the integrated equilibrium tail. This quantity thus emerges as the leading
    (first-order) renewal correction governing the asymptotic deviation in~\eqref{AISR2.2}.

    To access the next-order structure, assume that the second moment $\mu_2:=\mathbb{E}\tau^2$
    is finite, and let $\tau_F$ denote a random variable with distribution function $\mu_F$.
    Then
\begin{equation}       \label{AISR2.4}
    \mathbb{E}\tau_F={\int_{\mathbb{R}_+}{Q_F}(u)\,du}
    ={\int_{\mathbb{R}_+}u\,d\mu_F(u)}
    =\frac{1}{\,2\mu\,}\mathbb{E}\tau^2,
\end{equation}
    so that the integrated tail $R(t)$ in~\eqref{AISR2.3} converges to the finite limit
    $R(t)\to{\mu_2/(2\mu)}$ as $t\to\infty$. Substituting this asymptotic into
    the decomposition~{\eqref{AISR2.2}}, we arrive at the refined limit relation
\begin{equation}       \label{AISR2.5}
    U(t)-\frac{t}{\,\mu\,}\longrightarrow {\frac{\mu_2}{2\mu^2}} \qquad \text{as} \quad t\to\infty,
\end{equation}
    which is precisely Stone's~\cite{StoneAMS65} sharpening of the elementary renewal theorem.
    In view of~\eqref{AISR2.5}, we shall henceforth write
\[
    \mathcal{S}_\mu:=\frac{\mu_2}{2\mu^2}
\]
    and refer to $\mathcal{S}_\mu$ as the \emph{Stone constant}, which represents the universal
    second-order correction in the finite-variance regime.

    While relation~{\eqref{AISR2.5}} identifies the limiting constant of the renewal remainder, it provides
    no information about the rate at which this limit is approached. More generally, it is natural to investigate
    the finer asymptotic structure of renewal convolutions of the form $(g\ast U)(t)$, which lie at the heart of
    the key renewal theorem. Our first main result establishes a refined form of the key renewal theorem in the
    regime {\hyperref[AISR(iii)]{\textbf{(iii)}}}.

\begin{theorem}        \label{AISRTh1}
    Suppose that the inter-arrival distribution $F$ admits the representation {\eqref{AISR1.2}} corresponding
    to the regime {\hyperref[AISR(iii)]{\textbf{(iii)}}} and ${\mathcal{S}_\mu}:={\mu_2}\big/\left(2\mu^2\right)$
    denote the Stone constant. Then the following refined form of the key renewal theorem holds:
\begin{equation}     \label{AISR2.6}
    {\int_{[0,t]}Q_F(t-x)\,dU(x)}={\frac{1}{\,\mu\,}}R(t)+2{\mathcal{S}_\mu}\,Q_F(t)\!\big(1+o(1)\big)
    \qquad \text{as} \quad {t\to\infty},
\end{equation}
    where $R(t)$ is the leading renewal term, defined in {\eqref{AISR2.3}} and
\begin{equation}\label{AISR2.7}
    {Q_F}(t)={\frac{1}{\,\mu\beta\,}}t^{-\beta}\,L_Q(t),
\end{equation}
    therein $\beta=\alpha-1$ and $L_Q(t)\big/L(t)\to{1}$ as $t\to\infty$.
\end{theorem}

    The asymptotic representation obtained in {\hyperref[AISRTh1]{Theorem~\ref{AISRTh1}}} not only identifies
    the second-order structure of the renewal convolution, but also provides a refined description of the manner
    in which it approaches its limiting value. In particular, it reveals an explicit asymptotic decomposition of
    the deviation from the Stone constant $\mu_S$, separating the contribution of the accumulated tail from the
    local correction term. This is made precise in the following theorem. Define the tail of the leading renewal
    term as
\[
    \rho(t):=\int_{[t,\infty)}Q_F(u)\,du,
\]
    which represents the remaining mass of $Q_F$ beyond level $t$.

\begin{theorem}        \label{AISRTh2}
    Suppose that the conditions of {\hyperref[AISRTh1]{Theorem~\ref{AISRTh1}}} hold and set $\gamma:=\alpha-2$.
    Then the deviation of the renewal convolution from the Stone constant admits the following asymptote:
\begin{equation}     \label{AISR2.8}
    {\mathcal{S}_\mu}-{\int_{[0,t]}Q_F(t-x)\,dU(x)}
    ={\frac{1}{\,\mu\,}}\rho(t)-2{\mathcal{S}_\mu}\,Q_F(t)\!\big(1+o(1)\big)
\end{equation}
    as $t\to\infty$, where $\rho(t)$ is entirely governed, at the asymptotic level, by the tail
    of $Q_F(t)$, namely,
\begin{equation}\label{AISR2.9}
    {\frac{\rho(t)}{Q_F(t)}} \sim {\frac{1}{\,\gamma\,}}\,{t} \qquad \text{as} \quad {t\to\infty}.
\end{equation}
\end{theorem}

    As a direct consequence of {\hyperref[AISRTh2]{Theorem~\ref{AISRTh2}}}, we obtain the following
    explicit second-order asymptotic expansion for the renewal function.

\begin{corollary}        \label{AISRCor1}
    Under the conditions of {\hyperref[AISRTh1]{Theorem~\ref{AISRTh1}}},
    the renewal function admits the asymptotic expansion
\begin{equation}       \label{AISR2.10}
    {U(t)}=\frac{t}{\,\mu\,}+{\mathcal{S}_\mu}-{\frac{1}{\,\mu^2\beta\gamma\,}}\,{t^{-\gamma}}L_{\gamma}(t),
\end{equation}
    where $\beta=\alpha-1$, $\gamma=\alpha-2$ and $L_{\gamma}(t)\big/L(t)\to{1}$ as $t\to\infty$.
\end{corollary}

\begin{remark}        \label{AISRRem1}
    The above results show that, in the regime {\hyperref[AISR(iii)]{\textbf{(iii)}}}, the approach of
    the renewal function to its limiting linear form is governed by an additional heavy-tail correction
    of order $t^{-\gamma}L_\gamma(t)$, where $\gamma=\alpha-2$. This term reflects the residual influence
    of the heavy-tailed structure of the inter-arrival distribution beyond the classical finite-variance
    approximation.
\end{remark}

    When the inter-arrival distribution admits a density, the renewal measure is absolutely continuous, and
    its density $u(t)=U'(t)$ satisfies the classical renewal equation. While the renewal function $U(t)$ describes
    the cumulative structure of the renewal process, the density $u(t)$ reflects its local behaviour. Asymptotic
    results for $u(t)$ may therefore be viewed as a differential counterpart of those obtained for $U(t)$. We
    now turn to the asymptotic behaviour of the renewal density $u(t)$ under the same heavy-tailed framework
    and derive the corresponding second-order renewal asymptotics.

\subsection{Second-order asymptotics for the renewal density}       \label{IRSsec2.2}

    We now turn to the local counterpart of the results obtained for the renewal function. When the inter-arrival
    distribution admits a density, the renewal measure is absolutely continuous, and its density $u(t)=U'(t)$ describes
    the instantaneous renewal intensity. While the renewal function $U(t)$ captures the cumulative structure of the
    renewal process, the density $u(t)$ reflects its local behaviour on the time scale.

    Under the heavy-tailed framework considered above, one may expect that the second-order renewal structure obtained
    for $U(t)$ has a natural local analogue. The aim of this subsection is therefore to identify the corresponding
    second-order asymptotics for the renewal density $u(t)$. When $F$ admits a density $f$, the associated renewal
    density $u$ satisfies the classical renewal equation
\[
    u(t)=f(t)+(f\ast u)(t) \qquad \text{as} \quad {t\ge{0}}.
\]
    Passing to the Laplace-transform domain yields
\begin{equation}       \label{AISR2.11}
        \mathcal{L}u(s)=\frac{\mathcal{L}f(s)}{1-\mathcal{L}f(s)} {\raise1.5pt\hbox{,}}
\end{equation}
    where and throughout the paper, the symbol ``$\mathcal{L}{g}$'' denotes the (ordinary) Laplace transform (LT),
    defined for a locally integrable function $g$ by $\mathcal{L}g(s)=\int_{0}^{\infty} e^{-st}g(t)\,dt$.

    Representation \eqref{AISR2.11} shows that the asymptotic behaviour of the renewal density is governed by the same
    renewal singularity as that of the renewal function, but at a more local level. In particular, the analysis developed
    in {\hyperref[IRSec3]{Section~\ref{IRSec3}}} remains applicable in this setting and provides the basis for deriving
    second-order asymptotics for $u(t)$.

\begin{theorem}        \label{AISRTh3}
    Suppose that the inter-arrival distribution $F$ admits a density $f$ and the representation \eqref{AISR1.2}
    corresponding to the regime {\hyperref[AISR(iii)]{\textbf{(iii)}}}. Then the renewal density $u(t)=U'(t)$
    satisfies the asymptotic relation
\begin{equation}       \label{AISR2.12}
    u(t)=\frac{1}{\,\mu\,}+\frac{1}{\,\mu^2\beta\,}\,t^{-\beta}L_{\beta}(t),
\end{equation}
    where $\beta=\alpha-1$ and $L_{\beta}(t)\big/L(t)\to{1}$ as $t\to\infty$.
\end{theorem}

\subsection{Renewal Asymptotics at the Critical Case $\alpha=2$}   \label{IRSsec2.3}

    The case $\alpha=2$ represents a critical threshold in the renewal asymptotic structure, separating the
    finite-variance regime $\alpha>2$ from the infinite-variance regime $\alpha<2$. In contrast to the strictly
    finite-variance case, where moment conditions are determined solely by the power exponent, the finiteness
    of the second moment at this critical threshold depends on the integrability of the unction ${L(t)\big/t}$.

    In what follows, we restrict attention to the subregime $\mu_2<\infty$. More precisely, if
\[
    \overline{F}(t)=t^{-2}L(t),
\]
    then the second moment is finite if and only if the condition
\[
    \int^\infty \frac{L(t)}{t}\,dt<\infty
\]
    is satisfied. In particular, if $L(t)\to\infty$ or $L(t)\to c>0$, the above integral diverges, and
    hence $\mu_2=\infty$. Thus, at this critical threshold the variance is no longer determined by the
    power decay alone, but by the fine asymptotic behavior of the SV factor.

    This sensitivity renders the case $\alpha=2$ qualitatively distinct from both regimes $\alpha>2$
    and $\alpha<2$, and necessitates a separate asymptotic analysis.

    We now establish the renewal asymptotics at the critical threshold $\alpha=2$.

\begin{theorem}        \label{AISRTh4}
    Let the inter-arrival distribution $F$ admit the representation {\eqref{AISR1.2}} with $\alpha=2$. Then the
    deviation  of the renewal convolution from the Stone constant admits the following asymptotic representation:
\begin{equation}     \label{AISR2.14}
    {\mathcal{S}_\mu}-{\int_{[0,t]}Q_F(t-x)\,dU(x)}=
    {\frac{1}{\,\mu\,}}\rho(t)-{\frac{2\mathcal{S}_\mu}{\,\mu{t}\,}}\,{L_\alpha(t)},
\end{equation}
    with
\[
    \rho(t)\sim \frac{1}{\,\mu\,}\int_t^\infty \frac{L(u)}{u}\,du  \qquad \text{as} \quad  t\to\infty,
\]
    where $L_\alpha(t)\big/L(t)\to{1}$ as $t\to\infty$.
\end{theorem}

    As a direct consequence of {\hyperref[AISRTh4]{Theorem~\ref{AISRTh4}}}, we obtain the following explicit
    second-order asymptotic expansion for the renewal function.

\begin{corollary}        \label{AISRCor2}
    At the critical threshold $\alpha=2$ the following renewal asymptotic expansion holds:
\begin{equation}\label{eq:critical-U}
    U(t)=\frac{t}{\,\mu\,}+{\mathcal{S}_\mu} - \frac{1}{\,\mu\,}\,\rho(t)\bigl(1+o(1)\bigr)
    \qquad \text{as} \quad  t\to\infty
\end{equation}
    with $\rho(t)$ defined in {\hyperref[AISRTh4]{Theorem~\ref{AISRTh4}}}.
\end{corollary}

    The above results complete the description of the renewal asymptotics at the critical threshold $\alpha=2$.
    The following remark clarifies the structure of the corresponding second-order scale.

\begin{remark}[Critical tail scale and absence of universality]
    In the critical case $\alpha=2$, the second-order behaviour is governed by the integrated
    tail $\rho(t)\sim \int_t^\infty\left[L(u)\big/{\mu{u}}\right]du$ which should be viewed
    as the natural critical scale replacing the power-law asymptotics present in the
    regime {\hyperref[AISR(iii)]{\textbf{(iii)}}}. In contrast to that regime, no universal
    normalization of the form $c\,L(t)$ exists in general. The asymptotic behaviour of $\rho(t)$
    is intrinsically nonlocal and depends on the entire tail profile of $L$, rather than on its
    pointwise value at $t$. In this sense, the critical case exhibits a genuinely different
    asymptotic mechanism, where second-order corrections are determined by tail accumulation
    rather than local regular variation. In concrete logarithmic models, this nonlocal scale
    can nevertheless be evaluated explicitly. For instance, if
\[
    L(t)\sim\frac{1}{\log^{1+\delta}{t}} \qquad \text{as} \quad t\to\infty,
\]
    for some $\delta>0$, then
\[
    \rho(t)\sim \frac{1}{\mu\delta}\frac{1}{\log^{\delta}{t}} \qquad \text{as} \quad t\to\infty.
\]
    Such examples illustrate that the critical regime forms a boundary between power-law
    asymptotics and logarithmic corrections, with the integrated tail acting as the
    canonical second-order scale.
\end{remark}

    We now complete the critical-case analysis by establishing the corresponding asymptotic
    behaviour of the renewal density.

\begin{theorem}        \label{AISRTh5}
    Let the assumptions of {\hyperref[AISRTh4]{Theorem~\ref{AISRTh4}}} hold, and suppose further that
    the inter-arrival distribution $F$ admits a density $f$. Assume, in addition, that the remainder
    $u(t)-{1/\mu}$ is ultimately monotone as $t\to\infty$. Then the renewal density $u(t)$
    admits the asymptotic representation
\begin{equation}\label{AISR2.16}
    u(t)=\frac{1}{\,\mu\,}+\frac{1}{\mu^2}\,\frac{L_2(t)}{t}
    \qquad \text{as} \quad t\to\infty,
\end{equation}
    where $L_2(t)\big/L(t)\to1$ as $t\to\infty$.
\end{theorem}

    The proofs of the principal theorems are based on a refined LT analysis developed in the next section.

\section{Laplace-Transform Analysis}          \label{IRSec3}

    In this section, we develop the analytic framework underlying the general results. Our approach relies
    on a structural analysis of the renewal equation in the LT domain, which allows us to derive precise
    asymptotic information in the time domain.

    Let $F$ be nonlattice with $\mu:=\mathbb{E}\tau\in(0,\infty)$ and let $U$ denote the associated renewal
    function. Assuming $Q$ to be nonnegative and locally integrable on $[0,\infty)$, we define
\begin{equation}     \label{AISR3.1}
    \Delta_R(t):=(Q\ast{U})(t)-{\frac{1}{\,\mu\,}}R(t),
\end{equation}
    where $(Q\ast{U})(t):=\int_{[0,t]} Q(t-x)\,dU(x)$. The quantity $\Delta_R(t)$, renewal correction term,
    measures the deviation of the convolution $(Q\ast{U})(t)$ from the first-order renewal approximation as $t\to\infty$.

    The structure and asymptotics of the remainder $\Delta_R(t)$ become particularly transparent in the LT domain.
    We assume that the ordinary LT
\[
    \mathcal{L}Q(s):=\int_{[0,\infty)}e^{-su}Q(u)\,du
\]
    is finite for each $s\in\mathbb{R}_+$. By linearity of LT, we write
\begin{equation}     \label{AISR3.2}
    \mathcal{L}\Delta_R(s)=\mathcal{L}(Q\ast{U})(s)-{\frac{1}{\,\mu\,}}\mathcal{L}R(s).
\end{equation}

    Since $Q$ is nonnegative, Fubini's theorem justifies an interchange of the order of integration, and hence
\[
\begin{aligned}
    \mathcal{L}(Q\ast{U})(s)
    &=\int_{[0,\infty)}e^{-st}\left(\int_{[0,t]}Q(t-u)\,dU(u)\right)dt
\\
    &=\int_{[0,\infty)}\left(\int_{[u,\infty)}e^{-st}Q(t-u)\,dt\right)dU(u).
\end{aligned}
\]
    With the change of variables $v=t-u$, the inner integral becomes
\[
    \int_{[u,\infty)}e^{-st}Q(t-u)\,dt= e^{-su}\int_{[0,\infty)}e^{-sv}Q(v)\,dv= e^{-su}\mathcal{L}Q(s).
\]
    Substitution gives
\begin{equation}     \label{AISR3.3}
    \mathcal{L}(Q\ast{U})(s)=\mathcal{L}Q(s)\int_{[0,\infty)}e^{-su}\,dU(u)=\mathcal{L}Q(s)\cdot\widehat{U}(s),
\end{equation}
    where and throughout the paper, the hat symbol ``\,$\widehat{[\cdot]}$\,'' denotes the Laplace--Stieltjes
    transform (LST). In particular, $\widehat{U}(s):=\int_{[0,\infty)}e^{-st}\,dU(t)$ is the LST of $U(t)$.
    Next, applying Fubini's theorem once more, we obtain
\[
\begin{aligned}
    \mathcal{L}R(s)
    &=\int_{[0,\infty)}e^{-st}\left(\int_{[0,t)}Q(u)\,du\right)dt
\\
    &=\int_{[0,\infty)}\left(\int_{[u,\infty]}e^{-st}\,dt\right)Q(u)\,du
    =\frac{1}{\,s\,}\int_{[0,\infty)}{e^{-su}}Q(u)\,du,
\end{aligned}
\]
    and hence
\begin{equation}     \label{AISR3.4}
    \mathcal{L}R(s)=\frac{1}{\,s\,}\mathcal{L}Q(s).
\end{equation}
    Combining now {\eqref{AISR3.2}}---{\eqref{AISR3.4}}, we arrive at
\begin{equation}          \label{AISR3.5}
    \mathcal{L}\Delta_R(s)=\mathcal{L}Q(s)\left(\widehat U(s)-\frac{1}{\,\mu{s}\,}\right).
\end{equation}
    Finally, by the classical LST identity for the renewal function $\widehat{U}(s)\bigl(1-\widehat{F}(s)\bigr)=1$,
    equation {\eqref{AISR3.5}} can be altered as follows:
\[
    \mathcal{L}\Delta_R(s)=\mathcal{L}Q(s)\left(\frac{1}{1-\widehat F(s)}-{\frac{1}{\,\mu{s}\,}}\right).
\]
    This representation shows that the remainder term $\Delta_R(t)$ is governed by the cancellation of the principal
    singular term at $s=0$ in the LST of the renewal function. In other words, $\Delta_R(t)$ represents the residual
    contribution that remains after removing the leading term dictated by the elementary renewal theorem. From this
    viewpoint, the remainder $\Delta_R(t)$ reflects a Stone-type compensation phenomenon: once the intrinsic linear
    growth of the renewal function is subtracted, the remaining convolution term is naturally controlled by the scale
    of the kernel $Q$. Thus, the asymptotic behaviour of $\Delta_R(t)$ is governed not by the renewal function itself,
    but by the tail structure encoded in $Q$. This structural observation provides the guiding principle for the analysis
    below. In the next lemmas we make this mechanism precise and derive a quantitative estimate for the remainder $\Delta_R(t)$.

\subsection{LT Analysis in the Regime~{\hyperref[AISR(iii)]{\textbf{(iii)}}}}    \label{IRSec3.1}

    As before, denote by ${\mathcal{S}_\mu}:={\mu_2}\left/\left(2\mu^2\right)\right.$ the Stone constant.

\begin{lemma}    \label{AISRLem1}
    Let the inter-arrival distribution $F$ admit the representation {\eqref{AISR1.2}} corresponding to the
    regime {\hyperref[AISR(iii)]{\textbf{(iii)}}} with Stone constant ${\mathcal{S}_\mu}$.
    Then the following relation holds:
\begin{equation}          \label{AISR3.6}
    \widehat{U}(s)=\frac{1}{\mu{s}}+{\mathcal{S}_\mu}-{\frac{1}{\,\mu\,}}{\psi_{\alpha}(s)},
\end{equation}
    where
\begin{equation}\label{eq:psi-main}
    {\psi_{\alpha}(s)}={\frac{\pi}{\Gamma(\alpha)\sin(\pi\alpha)}}{\frac{1}{\,\mu\,}}s^{\alpha-2}L(1/s)
    \bigl(1+o(1)\bigr) \qquad \text{as} \quad {s\downarrow{0}};
\end{equation}
    in this and throughout the paper, $\Gamma(\cdot)$ denotes Euler's Gamma function.
\end{lemma}

\begin{proof}
    To analyse the behaviour of $\widehat U(s)$ near the origin, it is convenient to isolate the contribution of
    the tail of $F$. For this purpose we introduce the auxiliary function
\[
    B(s):=\mathcal{L}\overline{F}(s)-\mu+\frac{\mu_2}{2}s .
\]
    By the definition of the LT of the tail,
\[
    B(s)=\int_0^\infty\bigl(e^{-st}-1+st\bigr)\overline{F}(t)\,dt .
\]
    We now perform the rescaling $x=st$. This yields the representation
\[
    B(s)=s^{\alpha-1}L(1/s)\int_0^\infty\frac{e^{-x}-1+x}{x^{\alpha}}\frac{\overline{F}(x/s)}{(x/s)^{-\alpha}L(1/s)}\,dx.
\]
    The structure of this integral is particularly convenient, since the second factor in the integrand reflects
    precisely the regular variation of the tail. Indeed, because $\overline{F}\in\mathcal{RV}_{\infty}(-\alpha)$,
    for every fixed $x>0$ we have
\[
    \frac{\overline{F}(x/s)}{(x/s)^{-\alpha}L(1/s)}\longrightarrow{1} \qquad \text{as} \quad {s\downarrow{0}} .
\]
    We now justify the application of the dominated convergence theorem. To this end we invoke Potter's bounds
    for RV-functions;  see, e.g., \cite[Th~1.5.6]{BinghGT87}. For every $\varepsilon>0$ there exist
    constants $C>0$ and $s_0>0$ such that for all $s\in(0,s_0)$ and all $x>0$,
\[
    \frac{\overline{F}(x/s)}{(x/s)^{-\alpha}L(1/s)} \le{C}\max\left\{x^\varepsilon,x^{-\varepsilon}\right\}.
\]
    Consequently, the integrand is bounded by
\[
    C\,\frac{e^{-x}-1+x}{x^\alpha}\max\left\{x^\varepsilon,x^{-\varepsilon}\right\}.
\]
    This bound reflects the precise balance between the regular variation of the tail and the
    second-order cancellation encoded in the factor $e^{-x}-1+x$. In particular, it provides a
    uniform control of the integrand over the entire half-line $(0,\infty)$.

    We now verify that this bound is integrable on $(0,\infty)$. To this end, we analyse separately
    the behaviour near the origin and at infinity. Choose $\varepsilon>0$ sufficiently small and recall that
\[
    e^{-x}-1+x=\frac{\,x^2\,}{2}+\mathcal{O}\left(x^3\right) \qquad \text{as} \quad {x\downarrow{0}}.
\]
    Thus, the majorant behaves like $x^{2-\alpha-\varepsilon}$, which is integrable near the origin
    since $\alpha<3$. On the other hand,
\[
    e^{-x}-1+x\sim x \qquad \text{as} \quad {x\to\infty},
\]
    and therefore the majorant behaves like $x^{1-\alpha+\varepsilon}$, which is integrable
    at infinity because $\alpha>2$. Therefore, the above bound is integrable on $(0,\infty)$,
    and the dominated convergence theorem applies, yielding
\[
    B(s)\sim{s^{\alpha-1}}L(1/s)\int_0^\infty{\frac{e^{-x}-1+x}{x^{\alpha}}}\,dx
    \qquad \text{as} \quad {s\downarrow{0}}.
\]

    We now identify the constant appearing in this asymptotic representation. For $2<\alpha<3$,
    an integration by parts shows that
\[
    \int_0^\infty \frac{e^{-x}-1+x}{x^\alpha}\,dx=\Gamma(1-\alpha).
\]
    Consequently,
\begin{equation}\label{eq:Ltail-vs-Fhat}
    B(s)={\Gamma(1-\alpha)}s^{\alpha-1}L(1/s)+o\!\left(s^{\alpha-1}L(1/s)\right).
\end{equation}
    Next, relate this asymptotic expansion to the LT of the inter-arrival distribution function.
    Recall the classical LT identity
\begin{equation}\label{eq:LFbar111}
    \mathcal{L}\overline{F}(s)=\frac{1-\widehat{F}(s)}{s} \qquad \text{for} \quad {s>0}.
\end{equation}
    Substituting the definition of $B(s)$ into \eqref{eq:LFbar111}, we obtain
\begin{equation}\label{eq:LFbar-expansion}
    {1-\widehat{F}(s)}=sB(s)+{\mu}s-{\frac{\,\mu_2\,}{2}s^2}.
\end{equation}

    Combining \eqref{eq:Ltail-vs-Fhat} and \eqref{eq:LFbar-expansion}, we arrive at
\begin{equation}\label{eq:LFbar222}
    1-\widehat{F}(s)={\mu}s-{\frac{\,\mu_2\,}{2}}s^2+{\Gamma(1-\alpha)}s^\alpha{L(1/s)}
    +o\!\left(s^\alpha{L(1/s)}\right)
\end{equation}
    as $s\downarrow{0}$. At this point, the second-order structure becomes explicit: the principal
    linear term is accompanied by the finite-variance correction and the RV tail contribution.
    Applying Euler's reflection formula
\[
    \Gamma(\alpha){\Gamma(1-\alpha)}=\frac{\pi}{\sin(\pi\alpha)} {\raise1.5pt\hbox{,}}
\]
    the relation \eqref{eq:LFbar222} may be rewritten in the form
\begin{equation}\label{eq:LFbar-sss}
    \frac{1}{1-\widehat{F}(s)}=\frac{1}{{\mu}s}
    \frac{1}{1-\mu{\mathcal{S}_\mu}s+{s\psi_{\alpha}(s)}}{\raise1.5pt\hbox{.}}
\end{equation}
    where ${\psi_{\alpha}(s)}$ is defined in \eqref{eq:psi-main}.

    Now set
\[
    x(s):=-\mu{\mathcal{S}_\mu}s+{s\psi_{\alpha}(s)}.
\]
    Since ${s\psi_{\alpha}(s)}=\mathcal{O}\!\left(s^{\alpha-1}L(1/s)\right)=o(s)$ as $s\downarrow0$,
    we have $x(s)=\mathcal{O}(s)$. Therefore, the Taylor expansion of $(1+x)^{-1}$ at the origin yields
\begin{equation}\label{eq:LFbar-sds}
    \bigl(1+x(s)\bigr)^{-1}=1+\mu{\mathcal{S}_\mu}s-{s\psi_{\alpha}(s)}+\mathcal{O}\!\left(s^2\right)
    \qquad \text{as} \quad {s\downarrow{0}}.
\end{equation}
    By \eqref{eq:psi-main}, we have $s^2=o\!\left({s\psi_{\alpha}(s)}\right)$, and therefore
    the term $\mathcal{O}\!\left(s^2\right)$ is absorbed into the remainder ${s\psi_{\alpha}(s)}$.
    Finally, combining \eqref{eq:LFbar-sss} and \eqref{eq:LFbar-sds} with the classical LST identity
    for the renewal function,
\[
    \widehat{U}(s)\bigl(1-\widehat{F}(s)\bigr)=1,
\]
    we obtain the expansion {\eqref{AISR3.6}} with the tail term \eqref{eq:psi-main}.
    This completes the proof.
\end{proof}

    The next lemma identifies the small-$s$ asymptotic behaviour of the LT of the tail function $Q_F$
    of the equilibrium (stationary-excess) distribution defined in {\eqref{AISR2.1}}, a quantity that arises
    naturally in renewal theory and will play an important role in our subsequent analysis.

\begin{lemma}     \label{AISRLem2}
    Under the assumptions of {\hyperref[AISRLem1]{Lemma~\ref{AISRLem1}}},
\begin{equation}\label{eq:LQF-expansion}
    \mathcal{L}Q_F(s)={\mu{\mathcal{S}_\mu}}-{\psi_{\alpha}(s)},
\end{equation}
    where ${\psi_{\alpha}(s)}$ is given in {\eqref{eq:psi-main}}.
\end{lemma}

\begin{proof}
    By construction,
\[
    Q_F(t)=\frac{1}{\,\mu\,}\int_t^\infty \overline{F}(u)\,du.
\]
    Passing to LT and applying Fubini's theorem, we obtain
\[
    \mathcal{L}Q_F(s)=\frac{1}{\mu}
    \int_0^\infty{e^{-st}}\!\left(\int_t^\infty\overline{F}(u)\,du\right)\!dt
    =\frac{1}{{\mu}s}\int_0^\infty\left(1-e^{-su}\right)\overline{F}(u)\,du.
\]
    This identity reveals a convenient relation between $\mathcal{L}Q_F$ and the LT of the tail of $F$.
    Indeed, observing that
\[
    \int_0^\infty\left(1-e^{-su}\right)\overline{F}(u)\,du=\mu-\mathcal{L}\overline{F}(s),
\]
    we arrive at the representation
\[
    \mathcal{L}Q_F(s)=\frac{1}{{\mu}s}\bigl(\mu-\mathcal{L}\overline{F}(s)\bigr).
\]
    At this point we exploit the relations \eqref{eq:LFbar111} and \eqref{eq:LFbar222}.
    Then, as $s\downarrow{0}$,
\[
    \mu-\mathcal{L}\overline{F}(s)=\frac{\mu_2}{2}s-\Gamma(1-\alpha)s^{\alpha-1}L(1/s)
    +o\!\left(s^{\alpha-1}L(1/s)\right),
\]
    Substituting this expansion into the above identity and rewriting the coefficient of the second
    term by Euler's reflection formula, we arrive at the representation \eqref{eq:LQF-expansion}.
\end{proof}

    Thus the LT of the integrated tail exhibits the same heavy-tail correction structure that governs
    the expansion of the renewal transform.

    We are now in a position to determine the small-$s$ asymptotic behaviour of the LT of the renewal
    correction term $\Delta_R$.

\begin{lemma}\label{AISRLem3}
    Under the assumptions of {\hyperref[AISRLem1]{Lemma~\ref{AISRLem1}}}, as $s\downarrow0$,
\begin{equation}\label{eq:LDeltaR-expansion}
    \mathcal{L}\Delta_R(s)\sim\mu{{\mathcal{S}^2_\mu}}-2{\mathcal{S}_\mu}{\psi_{\alpha}(s)},
\end{equation}
    where ${\psi_{\alpha}(s)}$ is given in {\eqref{eq:psi-main}}.
\end{lemma}

\begin{proof}
    We begin with the representation obtained in {\eqref{AISR3.5}}. In the particular case $Q=Q_F$,
    it takes the form
\begin{equation}          \label{eq:LDR-expansion}
    \mathcal{L}\Delta_R(s)=\mathcal{L}Q_F(s)\left(\widehat{U}(s)-\frac{1}{\,\mu{s}\,}\right).
\end{equation}
    Next, we combine the relations {\eqref{AISR3.6}} and {\eqref{eq:LQF-expansion}}, which reveals that the
    entire second-order structure of the renewal transform is encoded in the LT of the equilibrium tail.
    In particular, the deviation of $\widehat{U}(s)$ from its leading term $1\big/(\mu{s})$ is precisely
    captured by $\mathcal{L}Q_F(s)$, and we obtain
\[
    \widehat{U}(s)-\frac{1}{\mu{s}}\sim{\frac{1}{\,\mu\,}}{\mathcal{L}Q_F(s)}.
\]
    Substituting this identity into \eqref{eq:LDR-expansion} yields the simplified representation
\begin{equation}          \label{eq:LDR-expansion1}
    \mathcal{L}\Delta_R(s)\sim{\frac{1}{\,\mu\,}}\bigl(\mathcal{L}Q_F(s)\bigr)^2.
\end{equation}

    Thus the asymptotic behaviour of $\mathcal{L}\Delta_R(s)$ is determined entirely by that
    of $\mathcal{L}Q_F(s)$. Applying the expansion \eqref{eq:LQF-expansion} to \eqref{eq:LDR-expansion1}
    immediately yields the asymptotic relation \eqref{eq:LDeltaR-expansion}.
\end{proof}

    We now identify the time-domain asymptotic behaviour of the remainder term $\Delta_R(t)$, with
    the equilibrium tail $Q_F$ serving as the governing comparison function. The following lemma
    establishes the fundamental relation between these quantities.

\begin{lemma}\label{AISRLem4}
    Assume that the conditions of {\hyperref[AISRLem1]{Lemma~\ref{AISRLem1}}} are satisfied. Let
\[
    \Delta_R(t):=(Q_F\ast{U})(t)-{\frac{1}{\,\mu\,}}R(t),
\]
    where $R(t)$ is the leading renewal term, as integrated equilibrium tail, defined in {\eqref{AISR2.3}}.
    Then the renewal correction term $\Delta_R(t)$ is asymptotically equivalent to the equilibrium tail, namely,
\begin{equation}\label{eq:DeltaR-QF-asymp}
    \Delta_R(t)\sim {2{\mathcal{S}_\mu}}\,Q_F(t)   \qquad \text{as} \quad t\to\infty.
\end{equation}
\end{lemma}

\begin{proof}
    We begin with the decomposition established in {\hyperref[AISRLem2]{Lemma~\ref{AISRLem2}}}, namely
\begin{equation}\label{eq:LQF-Mpsi}
    \mathcal{L}Q_F(s)=M-{\psi_{\alpha}(s)},
\end{equation}
    where $M:={\mu{\mathcal{S}_\mu}}$ and ${\psi_{\alpha}(s)}$ is given in {\eqref{eq:psi-main}}. Thus, near
    the origin, the LT of the equilibrium tail consists of a finite constant part and a singular correction
    term, and it is precisely this singular part that determines the large-time asymptotic behavior after
    Tauberian inversion.

    On the other hand, {\hyperref[AISRLem3]{Lemma~\ref{AISRLem3}}} shows that the renewal correction term
    admits the particularly convenient representation
\begin{equation}\label{eq:LDeltaR-square}
    \mathcal{L}\Delta_R(s)\sim\frac{1}{\,\mu\,}\bigl(\mathcal{L}Q_F(s)\bigr)^2.
\end{equation}
    Substituting \eqref{eq:LQF-Mpsi} into \eqref{eq:LDeltaR-square}, we obtain
\begin{equation}\label{eq:LDeltaR-full}
    \mathcal{L}\Delta_R(s)\sim\frac{\,M^2\,}{\mu}-\frac{\,2M\,}{\mu}{\psi_{\alpha}(s)}
    +\frac{1}{\,\mu\,}\bigl({\psi_{\alpha}(s)}\bigr)^2.
\end{equation}
    This expansion makes the underlying asymptotic mechanism transparent: the leading singular behavior
    of $\mathcal{L}\Delta_R(s)$ is generated by the term linear in ${\psi_\alpha}(s)$, whereas the
    quadratic term is of smaller order and therefore asymptotically negligible.

    Indeed, multiplying \eqref{eq:LQF-Mpsi} by $2M/\mu$ gives
\begin{equation}\label{eq:LQF-MpsiMu}
    \frac{\,2M\,}{\mu}\mathcal{L}Q_F(s)=\frac{\,2M^2\,}{\mu}-\frac{\,2M\,}{\mu}{\psi_{\alpha}(s)}.
\end{equation}
    Since $2<\alpha<3$, relation \eqref{eq:psi-main} implies ${\psi_{\alpha}(s)}\to{0}$ as $s\downarrow{0}$.
    Consequently, the singular parts of \eqref{eq:LDeltaR-full} and \eqref{eq:LQF-MpsiMu} coincide at the
    dominant order. Now, comparing \eqref{eq:LDeltaR-full} and \eqref{eq:LQF-MpsiMu}, we see that their
    singular parts at the origin coincide at the dominant order. More precisely,
\[
    \mathcal{L}\Delta_R(s)-\frac{\,M^2\,}{\mu}\sim
    {2{\mathcal{S}_\mu}}\mathcal{L}Q_F(s)-\frac{\,2M^2\,}{\mu} \qquad \text{as} \quad {s\downarrow{0}}.
\]
    The constants appearing on both sides of this asymptotic relation do not affect the large-$t$ behavior: in
    the Laplace domain they represent only contributions concentrated at $t=0$, whereas the asymptotic behavior
    under consideration is entirely governed by the singular terms at the origin. Therefore, a Tauberian transfer
    to the time domain yields the desired asymptotic relation \eqref{eq:DeltaR-QF-asymp}, thereby completing
    the proof.
\end{proof}

    {\hyperref[AISRLem4]{Lemma~\ref{AISRLem4}}} shows that the remainder term $\Delta_R(t)$ inherits 
    the same RV order as the tail of the equilibrium distribution. In particular, $\Delta_R(t)$ is
    asymptotically proportional to $Q_F(t)$, which reveals a simple and transparent structure of the 
    second-order term in the renewal convolution $(Q_F\ast{U})(t)$. This asymptotic relation will 
    play a key role in the subsequent analysis.

    The preceding lemmas establish a refined LST analysis of the renewal function $U(t)$, including sharp
    small-$s$ asymptotic expansions that encode the second-order renewal structure. We now extend this
    analysis to the renewal density $u(t)=U'(t)$, where the same mechanism persists at the level of
    the LT under the additional assumption that $F$ admits a density.

\begin{lemma}\label{AISRLem5}
    Suppose that the inter-arrival distribution $F$ satisfies representation {\eqref{AISR1.2}} corresponding
    to the regime {\hyperref[AISR(iii)]{\textbf{(iii)}}} with Stone constant ${\mathcal{S}_\mu}$ and admits
    a density $f$. Then there exists the renewal density $u(t)=U'(t)$ and
\begin{equation}\label{AISR3.26}
    \mathcal{L}u(s)=\frac{1}{\mu{s}}+\bigl({\mathcal{S}_\mu}-1\bigr)-\frac{1}{\,\mu\,}{\psi_{\alpha}(s)},
\end{equation}
    where ${\psi_{\alpha}(s)}$ is given in {\eqref{eq:psi-main}}.
\end{lemma}

\begin{proof}
    In our setting the renewal measure $U$ is absolutely continuous. Accordingly, the renewal function
    possesses a density $u(t):=U'(t)$, which satisfies the renewal equation $u=f+(f\ast{u})$. Passing
    to the LT domain, we obtain
\begin{equation}      \label{Luf(s)}
    \mathcal{L}u(s)=\frac{\mathcal{L}f(s)}{1-\mathcal{L}f(s)} {\raise1.5pt\hbox{.}}
\end{equation}
    At this point, it is important to note that, in the presence of a density, the LT of $f$ coincides
    with the LST of $F$, that is, $\mathcal{L}f(s)=\widehat{F}(s)$. Therefore, invoking the classical
    identity $\widehat{U}(s)\bigl(1-\widehat{F}(s)\bigr)=1$, we arrive at a more transparent representation
    of \eqref{Luf(s)}, namely,
\begin{equation}       \label{LuU(s)}
    \mathcal{L}u(s)=\widehat{U}(s)-1.
\end{equation}

    We are now in a position to conclude the argument. Indeed, substituting the representation \eqref{AISR3.6},
    established in {\hyperref[AISRLem1]{Lemma~\ref{AISRLem1}}}, into \eqref{LuU(s)} yields \eqref{AISR3.26}.
    This completes the proof.
\end{proof}


\subsection{LT Analysis at the Critical Case $\alpha=2$}  \label{IRSec3.2}

    This subsection develops the LT framework for the critical case $\alpha=2$, in parallel with the analysis
    carried out above for regime~{\hyperref[AISR(iii)]{\textbf{(iii)}}}. In contrast to the strictly power-type
    regime $2<\alpha<3$, the critical threshold $\alpha=2$ produces a borderline correction scale governed by
    the tail of the leading renewal term $\rho(t):=\int_{[t,\infty)}Q_F(u)\,du$. For this reason, the analysis
    below is organized around the asymptotic behavior of $Q_F$ and $\rho$, rather than around a direct critical
    analogue of {\hyperref[AISRLem4]{Lemma~\ref{AISRLem4}}}. In particular, the critical regime requires
    a reformulation of the LT analysis in terms of integrated tail quantities, which naturally capture
    the borderline correction scale.

\begin{lemma}\label{AISRLem-critical1}
    Let the inter-arrival distribution $F$ satisfy
\[
    \overline{F}(t)=t^{-2}L(t),
\]
    where $L\in\mathcal{SV}_\infty$, and assume that $\mu<\infty$ and $\mu_2<\infty$.  Then
\begin{equation}\label{eq:critical-rho}
    \rho(t)\sim \frac{1}{\,\mu\,}\int_t^\infty \frac{L(u)}{u}\,du   \qquad \text{as} \quad {t\to\infty}.
\end{equation}
\end{lemma}

\begin{proof}
    In present setting, Karamata's theorem for tail integrals implies
\[
    \int_t^\infty \overline{F}(u)\,du = \int_t^\infty \frac{L(u)}{u^2}\,du \sim\frac{L(t)}{t}
    \qquad \text{as} \quad {t\to\infty}.
\]
    Then the equilibrium-tail representation
\begin{equation}\label{eq:critical-QF-repr}
    Q_F(t)=1-\mu_F(t)\sim\frac{1}{\,\mu\,}\frac{L(t)}{t}  \qquad \text{as} \quad {t\to\infty}.
\end{equation}
    Substituting \eqref{eq:critical-QF-repr} into the definition of $\rho(t)$, we obtain \eqref{eq:critical-rho},
    provided the integral on the right-hand side is finite. But, in the present critical case,
\[
    \mu_2<\infty
    \qquad\Longleftrightarrow\qquad
    \int^\infty \frac{L(u)}{u}\,du<\infty.
\]
    Therefore the above tail integral is finite, and \eqref{eq:critical-rho} follows.
\end{proof}

    Thus, at the critical threshold $\alpha=2$, the equilibrium tail remains of order $L(t)/t$, while
    the genuine second-order correction scale is governed by its integrated tail. This transition
    from pointwise to integrated tail behavior is a distinctive feature of the critical regime.

\begin{lemma}\label{AISRLem-critical2}
    Let the assumptions of {\hyperref[AISRLem-critical1]{Lemma~\ref{AISRLem-critical1}}}
    hold and ${\mathcal{S}_\mu}$ is the Stone constant. Then
\begin{equation}\label{eq:critical-LQF-defect-L}
    \mu{\mathcal{S}_\mu}-\mathcal{L}Q_F(s) \sim
    \frac{1}{\,\mu\,}\int_{1/s}^\infty \frac{L(u)}{u}\,du
    \qquad \text{as} \quad {s\downarrow{0}}.
\end{equation}
\end{lemma}

\begin{proof}
    In the present setting,
\[
    \int_0^\infty Q_F(t)\,dt=\mu{\mathcal{S}_\mu}<\infty,
\]
    and therefore
\begin{equation}\label{eq:critical-LQF-start}
    \mu{\mathcal{S}_\mu}-\mathcal{L}Q_F(s)=\int_0^\infty \bigl(1-e^{-st}\bigr)Q_F(t)\,dt.
\end{equation}
    As it was proved in {\hyperref[AISRLem-critical1]{Lemma~\ref{AISRLem-critical1}}}, that
\[
    \rho(t)=\int_t^\infty Q_F(u)\,du \sim \frac{1}{\,\mu\,}\int_t^\infty \frac{L(u)}{u}\,du
    \qquad \text{as} \quad t\to\infty.
\]
    In particular, $Q_F\in\mathcal{RV}_\infty(-1)$. Moreover, since $Q_F$ is a tail function, it
    is nonincreasing, and by the finiteness of ${\mathcal S}_\mu$ it is integrable on $(0,\infty)$.

    We may therefore apply the borderline Abelian theorem for truncated Laplace integrals of integrable
    monotone tails. The asymptotic behavior is therefore driven by the tail of $Q_F$, and not by its local
    behavior near the origin. This reflects the fact that the kernel $(1-e^{-st})$ acts as a smooth cutoff
    at scale $t\sim 1/s$, so that the main contribution to the integral comes from the region $t$ of
    order $1/s$ and larger. Consequently,
\[
    \int_0^\infty \bigl(1-e^{-st}\bigr)Q_F(t)\,dt \sim \int_{1/s}^\infty Q_F(t)\,dt =\rho(1/s)
    \qquad \text{as} \quad s\downarrow{0}.
\]
    Substituting this asymptotic relation into \eqref{eq:critical-LQF-start}, we obtain
\[
    \mu{\mathcal{S}_\mu}-\mathcal{L}Q_F(s)\sim \rho(1/s).
\]
    Finally, using the asymptotic relation \eqref{eq:critical-rho} established in
    {\hyperref[AISRLem-critical1]{Lemma~\ref{AISRLem-critical1}}}, we obtain
\[
    \rho(1/s)\sim \frac{1}{\,\mu\,}\int_{1/s}^\infty \frac{L(u)}{u}\,du,
\]
    and hence \eqref{eq:critical-LQF-defect-L} follows. This identifies $\rho(1/s)$ as the precise
    defect of $\mathcal{L}Q_F(s)$ from its limiting value $\mu{\mathcal S}_\mu$.
\end{proof}

\begin{lemma}\label{AISRLem-critical3}
    Let the assumptions of {\hyperref[AISRLem-critical1]{Lemma~\ref{AISRLem-critical1}}} hold. Then
\begin{equation}\label{eq:critical-Lconv-main}
    \mu{\mathcal{S}_\mu}-\mu{s}\,\mathcal{L}(Q_F\ast{U})(s)
    \sim {\frac{1}{\,\mu\,}\int_{1/s}^\infty \frac{L(u)}{u}\,du}
    \qquad \text{as} \quad {s\downarrow{0}}.
\end{equation}
\end{lemma}

\begin{proof}
    We start from the standard LT identity for the renewal convolution:
\begin{equation}\label{eq:critical-Lconv-factor}
    \mathcal{L}(Q_F\ast{U})(s)=\widehat{U}(s)\mathcal{L}Q_F(s)\,.
\end{equation}
    Here $\widehat{U}$ is the LST of the renewal function as before. Since $\mu_2<\infty$, the inter-arrival
    transform admits the standard second-order expansion
\[
    1-\widehat F(s)=\mu s-\frac{\mu_2}{2}\,s^2+o\left(s^2\right)
    \qquad \text{as} \quad s\downarrow{0}.
\]
    Therefore, invoking the classical identity $\widehat{U}(s)\bigl(1-\widehat{F}(s)\bigr)=1$, we arrive at representation
\begin{equation}\label{eq:critical-Uhat-exp}
    \widehat U(s)=\frac{1}{\mu s}+{\mathcal{S}_\mu}+o(1)
    \qquad \text{as} \quad s\downarrow{0}.
\end{equation}
    On the other hand, by {\hyperref[AISRLem-critical2]{Lemma~\ref{AISRLem-critical2}}},
\begin{equation}\label{eq:critical-LQF-exp}
    \mathcal{L}Q_F(s)=\mu{\mathcal{S}_\mu}-\rho(1/s)\bigl(1+o(1)\bigr)
    \qquad \text{as} \quad s\downarrow{0},
\end{equation}
    where
\[
    \rho(t)=\int_t^\infty Q_F(u)\,du \sim \frac{1}{\,\mu\,}\int_t^\infty \frac{L(u)}{u}\,du
    \qquad \text{as} \quad t\to\infty.
\]
    Substituting \eqref{eq:critical-Uhat-exp} and \eqref{eq:critical-LQF-exp}
    into \eqref{eq:critical-Lconv-factor}, we obtain
\[
    \mathcal{L}(Q_F\ast U)(s)={\mathcal{S}_\mu}\frac{1}{\,s\,}-\frac{1}{\mu{s}}\,\rho(1/s)\bigl(1+o(1)\bigr)
    +\mathcal{O}\bigl(\rho(1/s)\bigr) + \mathcal{O}(1),
\]
    This expansion makes the asymptotic mechanism transparent: the dominant contribution comes from the term
    $\rho(1/s)\big/(\mu{s})$, while all remaining terms are of strictly smaller order. In particular, the
    renewal convolution inherits its second-order behavior directly from the defect of $\mathcal{L}Q_F(s)$.
    Indeed, as established in the proof of {\hyperref[AISRLem-critical1]{Lemma~\ref{AISRLem-critical1}}},
    that $\int_{1/s}^\infty\left[L(u)\big/u\right]du$ is a positive SV-function of $1/s$. In particular,
\[
    \frac{\rho(1/s)}{s}\longrightarrow \infty \qquad \text{as} \quad s\downarrow{0},
\]
    so that the terms $\mathcal{O}(\rho(1/s))$ and $\mathcal{O}(1)$ are negligible compared
    with $\rho(1/s)\big/{s}$.

    Multiplying the above expansion by $\mu{s}$, we therefore obtain
\[
    \mu s\,\mathcal{L}(Q_F\ast U)(s)=\mu{\mathcal{S}_\mu}
    -\rho(1/s)\bigl(1+o(1)\bigr) \qquad \text{as} \quad s\downarrow{0}
\]
    which is equivalent to \eqref{eq:critical-Lconv-main}. In particular, the renewal convolution
    inherits its second-order asymptotic behavior directly from the defect $\rho(1/s)$ in the LT of
    the equilibrium tail.
\end{proof}

    We now isolate the critical refinement of the renewal transform near the origin,
    which will serve as the key Laplace-domain input in the analysis of the renewal
    density at the threshold $\alpha=2$.

\begin{lemma}\label{AISRLem-critical-Uhat-refined}
    Let the assumptions of {\hyperref[AISRLem-critical2]{Lemma~\ref{AISRLem-critical2}}} hold, and define
\[
    h(s):=\mu{\mathcal S_\mu}-\mathcal{L}Q_F(s).
\]
    Then
\begin{equation}\label{eq:critical-Uhat-refined}
    \widehat{U}(s)-\frac{1}{\mu{s}}={\mathcal{S}_\mu}-\frac{1}{\,\mu\,}h(s)+o\!\bigl(h(s)\bigr)
    \qquad \text{as} \quad s\downarrow{0}.
\end{equation}
\end{lemma}

\begin{proof}
    We begin with the exact identity
\begin{equation}\label{eq:Uhat-LQF-exact}
    \widehat{U}(s)-\frac{1}{\mu{s}}=s\,\widehat{U}(s)\,\mathcal{L} Q_F(s),
\end{equation}
    which follows directly from the definitions. Indeed, using
\[
    \mathcal{L}Q_F(s)=\frac{\mu{s}-1+\widehat{F}(s)}{\mu{s^2}}
\]
    and the classical identity $\widehat{U}(s)\bigl(1-\widehat{F}(s)\bigr)=1$, we obtain
\[
    s\,\widehat{U}(s)\,\mathcal{L}Q_F(s)=\frac{\mu{s}-1+\widehat{F}(s)}{\mu{s}\bigl(1-\widehat{F}(s)\bigr)}
    =\widehat{U}(s)-\frac{1}{\mu{s}} {\raise1.5pt\hbox{,}}
\]
    proving \eqref{eq:Uhat-LQF-exact}.

    We next insert the asymptotic information already obtained for the two factors
    on the right-hand side. By the Stone expansion in the critical finite-variance case,
\[
    \widehat{U}(s)=\frac{1}{\mu{s}}+{\mathcal{S}_\mu}+o(1)  \qquad \text{as} \quad s\downarrow{0},
\]
    and therefore
\begin{equation}\label{eq:sUhat-critical}
    s\,\widehat{U}(s)=\frac{1}{\,\mu\,}+{\mathcal{S}_\mu}s+o(s).
\end{equation}
    On the other hand, by definition of $h(s)$,
\begin{equation}\label{eq:LQF-critical-h}
    \mathcal{L}Q_F(s)=\mu{\mathcal{S}_\mu}-h(s).
\end{equation}

    Substituting \eqref{eq:sUhat-critical} and \eqref{eq:LQF-critical-h} into \eqref{eq:Uhat-LQF-exact}, we obtain
\begin{eqnarray}  \label{eq:Uhat-expand-critical}
    \widehat{U}(s)-\frac{1}{\mu{s}}
    &=&\left(\frac{1}{\,\mu\,}+{\mathcal{S}_\mu}s+o(s)\right)
     \left(\mu{\mathcal{S}_\mu}-h(s)\right) \nonumber
\\
    &=&{\mathcal{S}_\mu}-\frac{1}{\,\mu\,}h(s)+\mu{\mathcal{S}^2_\mu}s+o(s)+\mathcal{O}\!\bigl(s\,h(s)\bigr).
\end{eqnarray}

    We now compare the scale $s$ with $h(s)$. By {\hyperref[AISRLem-critical2]{Lemma~\ref{AISRLem-critical2}}},
    $h(s)\sim \rho(1/s)$, where
\[
    \rho(t)\sim \frac{1}{\,\mu\,}\int_t^\infty \frac{L(u)}{u}\,du \qquad \text{as} \quad t\to\infty.
\]
    Since $L$ is SV and positive, we have
\[
    \rho(1/s)\ge \frac{1}{\,\mu\,}\int_{1/s}^{2/s}\frac{L(u)}{u}\,du
    \sim \frac{\log{2}}{\mu}\,L(1/s) \qquad \text{as} \quad s\downarrow{0}.
\]
    On the other hand, by a standard property of SV-functions, $s=o\!\bigl(L(1/s)\bigr)$, and therefore
\[
    s=o\!\bigl(\rho(1/s)\bigr)=o\!\bigl(h(s)\bigr) \qquad \text{as} \quad s\downarrow{0}.
\]
    Consequently, all terms of order $s$ in \eqref{eq:Uhat-expand-critical} are negligible compared with $h(s)$,
    and \eqref{eq:critical-Uhat-refined} follows. This completes the proof.
\end{proof}

    Next, we show that the renewal correction mechanism persists at the critical threshold $\alpha=2$.

\begin{lemma}\label{AISRLem-critical-DeltaR}
    Let the assumptions of {\hyperref[AISRLem-critical1]{Lemma~\ref{AISRLem-critical1}}} hold,
    and let ${\mathcal{S}_\mu}$ denote the Stone constant. Then the renewal correction term
\[
    \Delta_R(t):=(Q_F\ast{U})(t)-\frac{1}{\,\mu\,}\int_{[0,t]}Q_F(u)\,du
\]
    is asymptotically equivalent to the equilibrium tail $Q_F$, and
\begin{equation}\label{eq:DeltaR-Lasymp}
    \Delta_R(t)\sim \frac{2{\mathcal{S}_\mu}}{\mu}\,\frac{L(t)}{t}  \qquad \text{as} \quad t\to\infty.
\end{equation}
\end{lemma}

\begin{proof}
    We begin from the Laplace-domain representation of the renewal correction term:
\begin{equation}\label{eq:LDR-critical-start}
    \mathcal{L}\Delta_R(s)=\mathcal{L}Q_F(s)\left(\widehat{U}(s)-\frac{1}{\mu{s}}\right),
\end{equation}
    which is the critical counterpart of the structural identity used in the regime~{\hyperref[AISR(iii)]{\textbf{(iii)}}}.
    Thus, as before, the entire second-order correction is encoded in the interaction between the equilibrium tail
    transform $\mathcal{L}Q_F(s)$ and the Stone-type defect of the renewal transform.

    To isolate the leading constant contribution in the Laplace domain, we introduce the decomposition
\[
    M:=\mu{\mathcal{S}_\mu}  \qquad\text{and}\qquad  h(s):=M-\mathcal{L}Q_F(s).
\]
    Thus, the function $h(s)$ captures the entire deviation of the equilibrium tail transform from
    its limiting value $M$. By {\hyperref[AISRLem-critical2]{Lemma~\ref{AISRLem-critical2}}},
\begin{equation}\label{eq:h-critical}
    h(s)\sim \rho(1/s)  \qquad \text{as} \quad s\downarrow{0}.
\end{equation}
    In particular, $h(s)\to{0}$, so that $\mathcal{L}Q_F(s)=M-h(s)$, with a remainder term of order $\rho(1/s)$.

    Next, we recall that {\hyperref[AISRLem-critical-Uhat-refined]{Lemma~\ref{AISRLem-critical-Uhat-refined}}}
    provides the refined Stone expansion for the renewal transform near the origin, in which the defect
    function $h(s)$ with \eqref{eq:h-critical}, appears in the next-order correction of LST $\widehat{U}(s)$,
    so that we write
\begin{equation}\label{eq:critical-Uhat-ref}
    \widehat U(s)-\frac{1}{\mu{s}} ={\mathcal{S}_\mu}-\frac{1}{\,\mu\,}h(s)+o\!\bigl(h(s)\bigr)
    \qquad \text{as} \quad s\downarrow{0}.
\end{equation}
    This shows that, even at the critical threshold, the renewal transform and the equilibrium-tail
    transform remain coupled through the same defect function $h(s)$.

    Substituting $\mathcal{L}Q_F(s)=M-h(s)$ and \eqref{eq:critical-Uhat-ref}
    into \eqref{eq:LDR-critical-start}, we obtain
\[
    \mathcal{L}\Delta_R(s)=\bigl(M-h(s)\bigr)
    \left({\mathcal{S}_\mu}-\frac{1}{\,\mu\,}h(s)+o\!\bigl(h(s)\bigr)\right).
\]
    An expansion of the product yields
\begin{equation}\label{eq:LDR-critical-refined}
    \mathcal{L}\Delta_R(s) =\mu{\mathcal{S}^2_\mu}-2{\mathcal S_\mu}h(s)+o\!\bigl(h(s)\bigr)
    \qquad \text{as} \quad s\downarrow{0}.
\end{equation}
    On the other hand,
\[
    2{\mathcal{S}_\mu}\,\mathcal{L}Q_F(s)=2{\mathcal{S}_\mu}M-2{\mathcal{S}_\mu}h(s),
\]
    and therefore comparison with \eqref{eq:LDR-critical-refined} shows that the singular
    defects of $\mathcal{L}\Delta_R(s)$ and $2{\mathcal{S}_\mu}\,\mathcal{L}Q_F(s)$ coincide.
    In other words, the critical regime preserves exactly the same local renewal-compensation
    mechanism as in the case $2<\alpha<3$: only the global scale changes, whereas the local
    correction remains asymptotically proportional to the equilibrium tail itself.

    The constant terms in the above Laplace expansions affect only the mass concentrated at the
    origin and do not contribute to the large-$t$ asymptotics. Hence, by the same borderline
    Tauberian transfer used in the analysis of the critical defect, we obtain
\[
    \Delta_R(t)\sim {2\mathcal{S}_\mu}\,Q_F(t) \qquad \text{as} \quad t\to\infty.
\]
    Finally, invoking \eqref{eq:critical-QF-repr}, we conclude that
\[
    \Delta_R(t)\sim \frac{2\mathcal{S}_\mu}{\mu}\,\frac{L(t)}{t} {\raise1.5pt\hbox{,}}
\]
    which is exactly \eqref{eq:DeltaR-Lasymp}. This completes the proof.
\end{proof}

\subsection{Tauberian Transfer Mechanisms} \label{IRSec3.3}

    The results obtained in the Laplace domain encode the asymptotes of the renewal structure through
    their singular behavior near the origin. To transfer this information to the time domain, we rely
    on Tauberian principles, which have already been used repeatedly in the proofs of the preceding
    lemmas. These principles provide a systematic mechanism by which regular variation in the Laplace
    domain is converted into precise asymptotics in the time domain. In this subsection, we focus on
    a representative Tauberian mechanism that captures the type of transfer we use in the paper and
    formulate it explicitly in a form suited to our applications.

    The next two results are, in substance, Tauberian theorems describing this transfer mechanism.
    We formulate them as lemmas, reflecting their auxiliary role in our analysis, although they
    encode a general and independent principle of Tauberian transfer.

\begin{lemma}\label{AISRLem-Tauber}
    Let $g:\mathbb R_+\to\mathbb R_+$ be locally integrable, ultimately monotone,
    and $\int_0^\infty{g(t)}\,dt<\infty$. Suppose that, for some $\beta\in(1,2)$ and
    some $L\in\mathcal{SV}_\infty$,
\[
    g(t)\sim{K}\,t^{-\beta}L(t) \qquad \text{as} \quad {t\to\infty},
\]
    where $K$ is constant. Then
\begin{equation}\label{eq:defect-LT-Abelian}
    \mathcal{L}g(0)-\mathcal{L}g(s)\sim{-\Gamma(1-\beta)K}\,s^{\beta-1}L(1/s)
    \qquad \text{as} \quad {s\downarrow{0}}.
\end{equation}
    Conversely, if
\begin{equation}\label{eq:defect-LT-Tauber}
    \mathcal{L}g(0)-\mathcal{L}g(s)\sim{C}\,s^{\beta-1}L(1/s)
    \qquad \text{as} \quad {s\downarrow{0}}
\end{equation}
    for some $C>0$, then
\begin{equation}\label{eq:Tauber-back}
    g(t)\sim \frac{C}{-\Gamma(1-\beta)}\,t^{-\beta}L(t)
    \qquad \text{as} \quad {t\to\infty}.
\end{equation}
\end{lemma}

\begin{proof}
    Define the tail integral
\[
    H(t):=\int_t^\infty g(u)\,du.
\]
    Since $g\ge0$ and $\int_0^\infty g(u)\,du<\infty$, the function $H$ is finite, nonincreasing,
    and vanishes at infinity. Thus, $H$ captures the accumulated tail behaviour of $g$.
    By Karamata's theorem for tail integrals, we obtain
\begin{equation}\label{eq:H-asymp}
    H(t)\sim \frac{K}{\beta-1}\,t^{1-\beta}L(t) \qquad \text{as} \quad {t\to\infty},
\end{equation}
    since $\beta>1$.

    We next express the LT defect in terms of the tail. The elementary identity
\[
    \mathcal{L}g(0)-\mathcal{L}g(s)=\int_0^\infty\left(1-e^{-st}\right)g(t)\,dt,
\]
    can be rewritten, by Fubini's theorem, as
\begin{equation} \label{eq:defect-via-tail}
    \mathcal{L}g(0)-\mathcal{L}g(s)= s\int_0^\infty {e^{-st}H(t)}\,dt.
\end{equation}
    Thus, the Laplace defect is governed by the LT of the tail function $H$. Since $1-\beta\in(-1,0)$
    and $H$ monotone function, we may apply Karamata's theorem for LT to \eqref{eq:H-asymp}, yielding
\[
    \int_0^\infty e^{-st}H(t)\,dt \sim \frac{K}{\beta-1}\Gamma(2-\beta)\,s^{\beta-2}L(1/s)
    \qquad \text{as} \quad {s\downarrow{0}}.
\]
    Multiplying by $s$ and using \eqref{eq:defect-via-tail}, we obtain
\[
    \mathcal{L}g(0)-\mathcal{L}g(s) \sim \frac{K}{\beta-1}\Gamma(2-\beta)\,s^{\beta-1}L(1/s).
\]
    Finally, by the standard property of Gamma function $\Gamma(2-\beta)=(1-\beta)\Gamma(1-\beta)$,
    and hence we have
\[
    \frac{1}{\beta-1}\Gamma(2-\beta)=-\Gamma(1-\beta),
\]
    and \eqref{eq:defect-LT-Abelian} follows.

    Conversely, assume \eqref{eq:defect-LT-Tauber}. Then by \eqref{eq:defect-via-tail},
\[
    \int_0^\infty e^{-st}H(t)\,dt \sim{C}\,s^{\beta-2}L(1/s) \qquad \text{as} \quad {s\downarrow{0}}.
\]
    Since $H$ is nonincreasing and $2-\beta\in(0,1)$, Karamata's Tauberian theorem yields
\[
    H(t)\sim \frac{c}{\Gamma(2-\beta)}\,t^{1-\beta}L(t) \qquad \text{as} \quad {t\to\infty}.
\]
    Now, to recover an asymptote of $g$, we use the regular variation of $H$.
    Since $H\in\mathcal{RV}_{\infty}(1-\beta)$, it follows that ${tH'(t)\big/{H(t)}}\to{1-\beta}$
    which is immediate due to the Lamperti's theorem from {\cite[p.~59]{BinghGT87}}.
    This implies that $g(t)=-H'(t)$ is RV-function with index $-\beta$ and
\[
    g(t)\sim(\beta-1)\frac{C}{\Gamma(2-\beta)}\,t^{-\beta}L(t).
\]
    Using again $\Gamma(2-\beta)=(1-\beta)\Gamma(1-\beta)$, we conclude that
\[
    g(t)\sim \frac{C}{-\Gamma(1-\beta)}\,t^{-\beta}L(t),
\]
    which is exactly \eqref{eq:Tauber-back}. This completes the proof.
\end{proof}

    The following lemma records the Tauberian transfer at the critical borderline regime $\beta=1$.

\begin{lemma}\label{AISRLem-Tauber-critical}
    Let $g:\mathbb R_+\to\mathbb R_+$ be locally integrable, ultimately monotone, and $\int_0^\infty g(t)\,dt<\infty$.
    Assume that for some constant $K>0$,
\[
    g(t)\sim{K}\,\frac{L(t)}{t}  \qquad \text{as} \quad t\to\infty,
\]
    where $L\in\mathcal{SV}_\infty$. Then
\begin{equation}\label{eq:critical-defect-Abelian-final}
    {\mathcal{L}g(0)-\mathcal{L}g(s)}\sim{K}\int_{1/s}^\infty\frac{L(u)}{u}\,du
    \qquad \text{as} \quad s\downarrow{0}.
\end{equation}

    Conversely, if
\begin{equation}\label{eq:critical-defect-Tauber-final}
    {\mathcal{L}g(0)-\mathcal{L}g(s)}\sim{C}\int_{1/s}^\infty\frac{L(u)}{u}\,du
    \qquad \text{as} \quad s\downarrow{0}
\end{equation}
    for some $C>0$, then
\begin{equation}\label{eq:critical-Tauber-back-final}
    g(t)\sim{C}\,\frac{L(t)}{t} \qquad \text{as} \quad t\to\infty.
\end{equation}
\end{lemma}

\begin{proof}
    Define the tail integral
\[
    H(t):=\int_t^\infty g(u)\,du.
\]
    Since $g\ge0$ and $\int_0^\infty g(u)\,du<\infty$, the function $H$ is finite,
    nonincreasing, and  vanishes at infinity. Thus, $H$ represents the accumulated
    tail mass of $g$, and the critical Tauberian problem reduces to identifying the precise
    asymptotic relation between $H$ and the reference scale
\[
    \zeta_L(t):=\int_t^\infty \frac{L(u)}{u}\,du.
\]

    We first establish the Abelian direction. Since $g(t)\sim{KL(t)\big/t}$,
    the standard Karamata theorem for tail integrals in the borderline case yields
\begin{equation}\label{eq:H-critical-final}
    H(t)\sim {K}\,\zeta_L(t)
    \qquad \text{as} \quad t\to\infty.
\end{equation}
    In particular, since $\zeta_L(t)$ is a tail integral of a SV-function divided by $u$,
    it implies $\zeta_L(t)\in\mathcal{SV}_\infty$. Hence $H$ is also SV.

    We next express the LT defect in terms of the tail. The elementary identity
\[
    {\mathcal{L}g(0)-\mathcal{L}g(s)}=\int_0^\infty\left(1-e^{-st}\right)g(t)\,dt
\]
    can be rewritten, by Fubini's theorem, as
\begin{equation}\label{eq:defect-via-tail-critical-final}
    {\mathcal{L}g(0)-\mathcal{L}g(s)} =s\int_0^\infty{e^{-st}H(t)}\,dt.
\end{equation}

    The asymptotic behaviour of the right-hand side is determined by the contribution
    from the region $t\asymp{1/s}$. Since $H\in\mathcal{SV}_\infty$ is monotone up to
    the integral scale, we may apply the standard Tauberian principle for LT of monotone
    functions in the borderline case to obtain
\[
    \int_0^\infty{e^{-st}H(t)}\,dt \sim \int_0^{1/s}H(t)\,dt
    \qquad \text{as} \quad s\downarrow{0}.
\]
    Using \eqref{eq:H-critical-final}, we thus obtain
\[
    \int_0^\infty{e^{-st}H(t)}\,dt
    \sim{K}\int_0^{1/s}\left(\int_t^\infty \frac{L(u)}{u}\,du\right)dt.
\]
    Changing the order of integration, we arrive at
\[
    \int_0^{1/s}\left(\int_t^\infty \frac{L(u)}{u}\,du\right)dt
    =\int_0^\infty \frac{L(u)}{u}\bigl(\min\{u,1/s\}\bigr)\,du.
\]
    Splitting at $u=1/s$, this becomes
\[
    \int_0^{1/s}L(u)\,du+{\frac{1}{\,s\,}}\,{\zeta_L(1/s)}.
\]
    Since $L\in\mathcal{SV}_\infty$, the first term is of smaller order than the second,
    and therefore
\[
    \int_0^\infty {e^{-st}H(t)}\,dt \sim\frac{1}{\,s\,}H(1/s) \qquad \text{as} \quad s\downarrow{0}.
\]
    Multiplying by $s$ in \eqref{eq:defect-via-tail-critical-final}, and combining
    with \eqref{eq:H-critical-final}, we arrive at
\[
    {\mathcal{L}g(0)-\mathcal{L}g(s)} \sim {K}\,\zeta_L(1/s) \qquad \text{as} \quad s\downarrow{0},
\]
    which is precisely \eqref{eq:critical-defect-Abelian-final}.

    We now turn to the converse direction. Assume \eqref{eq:critical-defect-Tauber-final}.
    Then by \eqref{eq:defect-via-tail-critical-final},
\[
    \int_0^\infty{e^{-st}H(t)}\,dt
    \sim {C}\,\frac{1}{\,s\,}\,{\zeta_L(1/s)} \qquad \text{as} \quad s\downarrow{0}.
\]
    Since $H$ is nonincreasing and $\zeta_L(t)$ is SV, the Tauberian theorem for LT of SV-functions yields
\begin{equation}\label{eq:H-critical-conv-final}
    H(t)\sim{C}\,\zeta_L(t)  \qquad \text{as} \quad t\to\infty.
\end{equation}
    Further, since $g$ is ultimately monotone, there exists $t_0$ such that $g$ is nonincreasing
    on $[t_0,\infty)$. Fix $\lambda>1$. Then for all sufficiently large $t$,
\begin{equation}\label{eq:LeftRightBoundH}
    (\lambda-1)t\,g(\lambda{t}) \le \int_t^{\lambda{t}}g(u)\,du
    = H(t)-H(\lambda{t})\le (\lambda-1)t\,g(t).
\end{equation}
    Using \eqref{eq:H-critical-conv-final}, we obtain
\[
    H(t)-H(\lambda{t})\sim{C}\cdot\bigl(\zeta_L(t)-\zeta_L(\lambda{t})\bigr).
\]
    But
\[
    \zeta_L(t)-\zeta_L(\lambda{t})=\int_t^{\lambda{t}}\frac{L(u)}{u}\,du.
\]
    Since $L\in\mathcal{SV}_\infty$, the Uniform Convergence Theorem for SV-functions implies
\[
    \int_t^{\lambda{t}}\frac{L(u)}{u}\,du \sim {L(t)}\int_t^{\lambda{t}}\frac{du}{u}
    =L(t)\log\lambda  \qquad \text{as} \quad t\to\infty.
\]
    Therefore,
\begin{equation}\label{eq:H(t)Asymp}
    H(t)-H(\lambda{t})\sim{C}\,L(t)\log\lambda.
\end{equation}

    Combining this with the monotonicity bounds above, we first obtain
\[
    \liminf_{t\to\infty}\frac{t\,g(t)}{L(t)} \ge \frac{C\log\lambda}{\lambda-1}{\raise1.5pt\hbox{.}}
\]
    To derive the corresponding upper bound, we use the left-hand side of {\eqref{eq:LeftRightBoundH}}
    and the asymptotic relation {\eqref{eq:H(t)Asymp}}, and replacing $t$ by $t/\lambda$, we obtain
\[
    g(t)\lesssim \frac{C\log\lambda}{\lambda-1}\,
    \frac{L(t/\lambda)}{t/\lambda}{\raise1.5pt\hbox{.}}
\]
    Since $L\in\mathcal{SV}_\infty$, the Uniform Convergence Theorem implies ${L(t/\lambda)}{L(t)}\to{1}$
    as $t\to\infty$, and therefore
\[
    g(t)\lesssim \frac{C\lambda\log\lambda}{\lambda-1}\,\frac{L(t)}{t} {\raise1.5pt\hbox{.}}
\]
    Consequently,
\[
    \limsup_{t\to\infty}\frac{t\,g(t)}{L(t)}
    \le \frac{C\lambda\log\lambda}{\lambda-1} {\raise1.5pt\hbox{.}}
\]

    Thus, letting $\lambda\downarrow{1}$, and using the elementary
    relation ${\log\lambda}\big/(\lambda-1)\to{1}$, we conclude that
\[
    \liminf_{t\to\infty}\frac{t\,g(t)}{L(t)}\ge C
    \qquad \text{and} \qquad
    \limsup_{t\to\infty}\frac{t\,g(t)}{L(t)}\le C.
\]
    Hence
\[
    g(t)\sim C\,\frac{L(t)}{t}
    \qquad \text{as} \quad t\to\infty,
\]
    which proves \eqref{eq:critical-Tauber-back-final}. This completes the proof.
\end{proof}

    The LT analysis for both regime~{\hyperref[AISR(iii)]{\textbf{(iii)}}} and the critical
    case $\alpha=2$ is now complete, and we are ready to turn to the proof of the Principal Results.

\section{Proofs of Principal Results}   \label{IRSec4}

    This section is devoted to the proofs of the principal results. Building upon the LT analysis
    developed in the previous section, we exploit the fact that the Laplace-domain framework encodes
    the asymptotic structure of the equilibrium tail and governs its propagation through the renewal
    convolution. This perspective reveals a precise transfer mechanism by which RV of the tail
    induces the corresponding renewal asymptotics, allowing for a sharp identification of both
    the leading term and its second-order refinement.

\begin{proof}[{Proof of {\hyperref[AISRTh1]{Theorem~\ref{AISRTh1}}}}]
    We start from the canonical renewal-convolution decomposition associated with the equilibrium tail:
\begin{equation}\label{T1-decomp}
    (Q_F\ast{U})(t)=\frac{1}{\,\mu\,}R(t)+\Delta_R(t),
\end{equation}
    where $R(t)$ is the leading renewal term, and $\Delta_R(t)$ is the renewal correction term.
    This decomposition serves as the fundamental organizing principle of the argument: once the leading
    first-order renewal contribution $R(t)/\mu$ is separated off, the entire second-order structure
    of the convolution $(Q_F\ast U)(t)$ is encoded in the remainder term $\Delta_R(t)$.

    We next identify the intrinsic asymptotic scale of the kernel $Q_F(t)=1-\mu_F(t)$. Under the standing
    assumption that the tail $\overline{F}\in\mathcal{RV}_\infty(-\alpha)$, the asymptotic behavior is
    entirely governed by RV. Therefore, by Karamata's theorem for integrated RV-functions,
\begin{equation}\label{T1-QF-asymp}
    Q_F(t)\sim \frac{1}{\mu\beta}\,t^{-\beta}L(t) \qquad \text{as} \quad t\to\infty,
\end{equation}
    that is, $Q_F\in\mathcal{RV}_\infty(-\beta)$, where $\beta:=\alpha-1$.

    To obtain a representation stable under asymptotic manipulations, we absorb negligible fluctuations
    into a SV factor $L_Q$, rewriting \eqref{T1-QF-asymp} in the form
\begin{equation}\label{T1-QF-LQ}
    Q_F(t)=\frac{1}{\mu\beta}\,t^{-\beta}L_Q(t),
\end{equation}
    where $L_Q(t)\big/L(t)\to{1}$ as $t\to\infty$. The representation \eqref{T1-QF-LQ} isolates the
    principal RV scale while preserving full asymptotic equivalence, and coincides with \eqref{AISR2.7}.

    The final step is to identify the precise asymptotic contribution of the remainder term $\Delta_R(t)$,
    which constitutes the core second-order mechanism in the renewal structure. This is provided by
    {\hyperref[AISRLem4]{Lemma~\ref{AISRLem4}}}, showing that in the present finite-variance
    RV regime~{\hyperref[AISR(iii)]{\textbf{(iii)}}},
\begin{equation}\label{T1-Delta-asymp}
    \Delta_R(t)\sim {2\mathcal{S}_\mu}\,Q_F(t)  \qquad \text{as} \quad t\to\infty,
\end{equation}
    where ${\mathcal{S}_\mu}$ is the Stone constant. This relation reveals a remarkable structural fact:
    the entire second-order contribution is asymptotically proportional to the equilibrium tail itself.
    In particular, no additional smoothing or averaging occurs at this level, and the renewal correction
    inherits the same RV scale as $Q_F$.

    Substituting \eqref{T1-Delta-asymp} into \eqref{T1-decomp}, we obtain
\[
    (Q_F\ast{U})(t)\sim\frac{1}{\,\mu\,}R(t)+2{\mathcal{S}_\mu}\,Q_F(t)
    \qquad \text{as} \quad t\to\infty,
\]
    which completes the identification of the refined asymptotic structure
    and yields precisely {\eqref{AISR2.6}}.

    Hence both \eqref{AISR2.6} and \eqref{AISR2.7} follow, and the proof is complete.
\end{proof}

\begin{proof}[{Proof of {\hyperref[AISRTh2]{Theorem~\ref{AISRTh2}}}}]
    Recall the equilibrium (stationary-excess) distribution associated with $F$, defined in {\eqref{AISR2.1}},
    and let $\tau_F$ be a random variable with distribution function $\mu_F$. By relation~{\eqref{AISR2.4}}
    we indeed have $\mathbb{E}\tau_F\in(0,\infty)$. Then, the Stone constant we may write in the form
\begin{equation}\label{T2-Stone}
    {\mathcal{S}_\mu}:={\frac{1}{\,\mu\,}}\mathbb{E}\tau_F
    =\frac{1}{\,\mu\,}\int_{\mathbb{R}_+}Q_F(u)\,du.
\end{equation}
    Thus, ${\mathcal{S}_\mu}$ is precisely the normalized total mass of the equilibrium tail. Next,
    decomposing the tail via the tail of the leading renewal term $\rho(t):=\int_{[t,\infty)}Q_F(u)\,du$,
    relation {\eqref{T2-Stone}} yields the exact identity
\begin{equation}\label{T2-Stone-decomp}
    {\mathcal{S}_\mu}={\frac{1}{\,\mu\,}}R(t)+{\frac{1}{\,\mu\,}}\rho(t),
\end{equation}
    where $R(t)=\int_{[0,t]}Q_F(u)\,du$ is the leading renewal term, as before.

    We now introduce the deviation from the Stone limit by
\begin{equation}\label{T2-Delta}
    \Delta(t):={\mathcal{S}_\mu}-(Q_F\ast{U})(t).
\end{equation}
    Since ${Q_F}$ is nonnegative and integrable on $\mathbb{R}_+$ (indeed,
    $\int_{\mathbb{R}_+}{Q_F}(u)\,du\in(0,\infty)$ under regime~{\hyperref[AISR(iii)]{\textbf{(iii)}}}),
    the key renewal theorem applied to ${Q_F}$ yields $\Delta(t)\to{0}$ as $t\to\infty$.
    Substituting \eqref{T2-Stone-decomp} into \eqref{T2-Delta}, and then adding and subtracting
    $R(t)\big/{\mu}$, we obtain
\begin{equation}\label{T2-basic-decomp}
    \Delta(t)
    ={\frac{1}{\,\mu\,}}\rho(t)-\Delta_R(t),
\end{equation}
    where $\Delta_R(t)$ is the renewal correction term defined in {\eqref{AISR3.1}}. This decomposition
    is the essential structural step: the deviation from the Stone constant splits into a tail contribution,
    represented by $\rho(t)\big/{\mu}$, and a renewal correction term, represented by $\Delta_R(t)$.

    We next identify the asymptotic scale of $\rho(t)$. By {\hyperref[AISRTh1]{Theorem~\ref{AISRTh1}}},
    the equilibrium tail satisfies
\[
    Q_F(t)=\frac{1}{\mu\beta}\,t^{-\beta}L_Q(t),
\]
    where $\beta=\alpha-1\in(1,2)$ and $L_Q\in\mathcal{SV}_\infty$ with $L_Q(t)\big/L(t)\to1$. Hence
    $Q_F\in\mathcal{RV}_\infty(-\beta)$. Since $\beta>1$, Karamata's theorem for tail integrals
    of RV-functions gives
\begin{equation}\label{T2-rho-asymp}
    \rho(t)\sim \frac{1}{\beta-1}\,tQ_F(t)=\frac{1}{\,\gamma\,}\,tQ_F(t)
    \qquad \text{as} \quad t\to\infty,
\end{equation}
    where $\gamma:=\beta-1=\alpha-2$. This is precisely the asymptotic relation asserted in \eqref{AISR2.9}.
    This step is crucial: it shows that the dominant contribution is governed by the integrated tail,
    rather than by the pointwise behaviour of $Q_F$. In particular, the scale $t Q_F(t)$ emerges
    naturally as the correct second-order normalization.

    Hence, the leading contribution to $\Delta(t)$ is governed by the asymptotic behavior of the tail
    of the leading renewal term, while $\Delta_R(t)$ represents the renewal approximation error
    associated with the integral $R(t)$. It suffices to verify that this correction term is of
    smaller order. By {\hyperref[AISRLem4]{Lemma~\ref{AISRLem4}}},
\begin{equation}\label{T2-DeltaR-asymp}
    \Delta_R(t)\sim {2\mathcal{S}_\mu}\,Q_F(t)  \qquad \text{as} \quad t\to\infty.
\end{equation}
    Substituting \eqref{T2-rho-asymp} and \eqref{T2-DeltaR-asymp} into \eqref{T2-basic-decomp},
    we conclude that
\[
    {\mathcal{S}_\mu}-(Q_F\ast U)(t)
    ={\frac{1}{\,\mu\,}}\rho(t)-2{\mathcal{S}_\mu}\,Q_F(t)\bigl(1+o(1)\bigr)
    \qquad \text{as} \quad t\to\infty.
\]
    This is precisely the asymptotic representation claimed in \eqref{AISR2.8}.
    In view of \eqref{T2-rho-asymp}, this completes the proof.
\end{proof}

\begin{proof}[{Proof of {\hyperref[AISRCor1]{Corollary~\ref{AISRCor1}}}}]
    The result follows by combining the second-order renewal asymptotics established in
    Theorems~\ref{AISRTh1} and~\ref{AISRTh2}. Rearranging {\eqref{AISR2.8}}, we write
\begin{equation}     \label{C1-QFconv}
    (Q_F\ast{U})(t)={\mathcal{S}_\mu}-{\frac{1}{\,\mu\,}}\rho(t)+
    2{\mathcal{S}_\mu}\,Q_F(t)\!\big(1+o(1)\big)  \qquad \text{as} \quad t\to\infty,
\end{equation}
    where $\rho(t):=\int_{[t,\infty)}Q_F(u)\,du$ and by {\hyperref[AISRTh1]{Theorem~\ref{AISRTh1}}},
    the equilibrium tail satisfies
\[
    {Q_F}(t)={\frac{1}{\,\mu\beta\,}}t^{-\beta}\,L_Q(t),
\]
    where $\beta=\alpha-1\in(1,2)$ and $L_Q(t)\big/L(t)\to{1}$ as $t\to\infty$. Consequently,
\begin{equation}\label{C1-rho-asymp}
    \rho(t)\sim \frac{1}{\mu\beta\gamma}\,t^{-\gamma}L_Q(t) \qquad \text{as} \quad t\to\infty,
\end{equation}
    where $\gamma=\alpha-2\in(0,1)$. Since $\beta>1$, we have $\gamma=\beta-1\in(0,1)$, and therefore
\[
    \frac{Q_F(t)}{\rho(t)} \sim \frac{t^{-\beta}}{t^{-\gamma}}
    =\frac{\,1\,}{t} \longrightarrow{0} \qquad \text{as} \quad t\to\infty
\]
    which shows that the term involving $Q_F(t)$ is asymptotically negligible compared to $\rho(t)$.
    Therefore, recalling the renewal decomposition {\eqref{AISR2.2}}, and combining it with {\eqref{C1-QFconv}}
    and {\eqref{C1-rho-asymp}} we have
\begin{equation}\label{C1-U-decomp}
    U(t)=\frac{t}{\,\mu\,}+{\mathcal{S}_\mu}-\frac{1}{\,\mu\,}\rho(t)+o(\rho(t))
    \qquad \text{as} \quad t\to\infty.
\end{equation}
    Thus, the second-order asymptotic structure of the renewal function is entirely inherited from
    the integrated tail $\rho(t)$. The contribution of the local term $Q_F(t)$ becomes negligible
    at this scale, which reflects the smoothing effect of integration inherent in the renewal equation.
    Substituting {\eqref{C1-rho-asymp}} into {\eqref{C1-U-decomp}}, and absorbing the negligible error
    term into the SV factor, we conclude that
\[
    U(t)=\frac{t}{\,\mu\,}+{\mathcal{S}_\mu}-\frac{1}{\mu^2\beta\gamma}\,t^{-\gamma}L_\gamma(t),
\]
    where $L_{\gamma}(t)\big/L(t)\to{1}$ as $t\to\infty$; this is exactly \eqref{AISR2.10}.
    The proof is complete.
\end{proof}

\begin{proof}[{Proof of {\hyperref[AISRTh3]{Theorem~\ref{AISRTh3}}}}]
    Set
\[
    v(t):=u(t)-\frac{1}{\,\mu\,}{\raise1.5pt\hbox{,}}
\]
    so that $v(t)$ isolates the fluctuation of the renewal density around its limiting level.
    The argument proceeds by extracting the precise small-$s$ defect in the Laplace domain
    and transferring it to the time domain via a Tauberian mechanism.

    We begin from the refined Laplace-domain expansion established in
    {\hyperref[AISRLem5]{Lemma~\ref{AISRLem5}}}:
\begin{equation}    \label{eq:Lnu-Th3}
    \mathcal{L}v(s)=\mathcal{L}u(s)-\frac{1}{\mu{s}}
    =\bigl({\mathcal{S}_\mu}-1\bigr)-\frac{1}{\,\mu\,}\psi_{\alpha}(s),
\end{equation}
    where
\begin{equation}\label{eq:psiTh3-final}
    \psi_{\alpha}(s)=\frac{\pi}{\Gamma(\alpha)\sin(\pi\alpha)}\,
    \frac{1}{\,\mu\,}s^{\alpha-2}L(1/s)\bigl(1+o(1)\bigr)
    \qquad \text{as} \quad s\downarrow{0}.
\end{equation}

    Since $\psi_\alpha(s)\to0$ as $s\downarrow0$, relation~\eqref{eq:Lnu-Th3}
    identifies the total mass of $v$ as
\[
    \mathcal{L}v(0):=\int_0^\infty v(t)\,dt=\mathcal{S}_\mu-1,
\]
    so that the entire second-order deviation from equilibrium is encoded in the
    vanishing defect term $\psi_\alpha(s)$. Equivalently,
\begin{equation}\label{eq:defect-v-Th3}
    \mathcal{L}v(0)-\mathcal{L}v(s)=\frac{1}{\,\mu\,}\psi_\alpha(s),
\end{equation}
    which isolates the small-$s$ asymptotic structure responsible for the tail behaviour of $v$.

    Noting that $\alpha-2=\beta-1\in(0,1)$, substitution of~\eqref{eq:psiTh3-final}
    into~\eqref{eq:defect-v-Th3} yields the canonical Tauberian form
\[
    \mathcal{L}v(0)-\mathcal{L}v(s)\sim c\,s^{\beta-1}L(1/s),
\]
    where
\[
    c:=\frac{\pi}{\Gamma(\alpha)\sin(\pi\alpha)}\,\frac{1}{\mu^2}{\raise1.5pt\hbox{.}}
\]

    The passage to the time domain now follows via a Tauberian transfer.
    Since $v$ is ultimately monotone by assumption,
    {\hyperref[AISRLem-Tauber]{Lemma~\ref{AISRLem-Tauber}}} applies and yields
\begin{equation}\label{eq:v-preconst-Th3}
    v(t)\sim \frac{c}{-\Gamma(1-\beta)}\,t^{-\beta}L(t)
    \qquad \text{as} \quad t\to\infty,
\end{equation}
    thereby identifying the precise second-order asymptotic behaviour of the renewal density.
    It remains to express the coefficient in explicit form. To this end, we simplify the constant
    in~\eqref{eq:v-preconst-Th3}. Observing that $\Gamma(1-\beta)=-\beta\Gamma(1-\alpha)$
    and invoking Euler's reflection formula, we obtain
\[
    \frac{c}{-\Gamma(1-\beta)}=\frac{1}{\mu^2\beta}{\raise1.5pt\hbox{.}}
\]
    Substituting this expression into~\eqref{eq:v-preconst-Th3} yields the explicit asymptotic
\[
    v(t)\sim \frac{1}{\mu^2\beta}\,t^{-\beta}L(t)  \qquad \text{as} \quad t\to\infty.
\]
    Recalling the definition $v(t)=u(t)-1\big/\mu$, we conclude that
\[
    u(t)=\frac{1}{\,\mu\,}+\frac{1}{\mu^2\beta}\,t^{-\beta}L_\beta(t),
\]
    where $L_\beta(t)\big/L(t)\to1$ as $t\to\infty$. This establishes~\eqref{AISR2.12}
    and completes the proof.
  \end{proof}

\begin{proof}[{Proof of {\hyperref[AISRTh4]{Theorem~\ref{AISRTh4}}}}]
    We start from the canonical decomposition associated with the equilibrium tail. Recall
    that ${\mu\mathcal{S}_\mu}=\int_{\mathbb R_+}Q_F(u)\,du$, and $\rho(t):=\int_{[t,\infty)}Q_F(u)\,du$.
    Then we have the exact identity
\begin{equation}\label{eq:T4-Stone-decomp}
    {\mathcal{S}_\mu}=\frac{1}{\,\mu\,}\int_{[0,t]}Q_F(u)\,du+\frac{1}{\,\mu\,}\rho(t).
\end{equation}

    On the other hand, by the very definition of the renewal correction term,
\[
    (Q_F\ast{U})(t)=\frac{1}{\,\mu\,}\int_{[0,t]}Q_F(u)\,du+\Delta_R(t),
\]
    where
\[
    \Delta_R(t):=(Q_F\ast{U})(t)-\frac{1}{\,\mu\,}\int_{[0,t]}Q_F(u)\,du.
\]
    Comparing this representation with \eqref{eq:T4-Stone-decomp}, we arrive at the fundamental decomposition
\begin{equation}\label{eq:T4-main-decomp}
    {\mathcal{S}_\mu}-(Q_F\ast{U})(t)=\frac{1}{\,\mu\,}\rho(t)-\Delta_R(t).
\end{equation}

    Thus, the deviation from the Stone constant splits naturally into two components:
    the accumulated tail contribution $\rho(t)$ and the local renewal correction term $\Delta_R(t)$.
    The critical case $\alpha=2$ is characterized precisely by the fact that these two terms live
    on different asymptotic scales.

    We first identify the dominant contribution. By {\hyperref[AISRLem-critical1]{Lemma~\ref{AISRLem-critical1}}},
\[
    \rho(t)\sim \frac{1}{\,\mu\,}\int_t^\infty \frac{L(u)}{u}\,du
    \qquad \text{as} \quad t\to\infty.
\]
    In particular, the leading correction is no longer of pure power-law type; instead,
    it is governed by the integrated SV tail, which reflects the borderline
    nature of the threshold $\alpha=2$.

    We next turn to the local renewal correction term. By
    {\hyperref[AISRLem-critical-DeltaR]{Lemma~\ref{AISRLem-critical-DeltaR}}},
\[
    \Delta_R(t)=\frac{2\mathcal{S}_\mu}{\mu{t}}\,L_\alpha(t),
\]
    where $L_\alpha\in\mathcal{SV}_\infty$ such that $L_\alpha(t)\big/L(t)\to1$ as $t\to\infty$.
    Substituting the above asymptotics into \eqref{eq:T4-main-decomp}, we obtain
\[
    {\mathcal{S}_\mu}-(Q_F\ast{U})(t)
    =\frac{1}{\,\mu\,}\rho(t)-\frac{2\mathcal{S}_\mu}{\mu{t}}\,L_\alpha(t),
\]
    which is precisely \eqref{AISR2.14}. This completes the proof.
\end{proof}

\begin{proof}[{Proof of {\hyperref[AISRCor2]{Corollary~\ref{AISRCor2}}}}]
    The conclusion is an immediate consequence of the critical convolution asymptotics
    established in {\hyperref[AISRTh4]{Theorem~\ref{AISRTh4}}}, once it is combined with the
    renewal decomposition {\eqref{AISR2.2}}. Indeed, rearranging \eqref{AISR2.14}, we obtain
\begin{equation}\label{eq:critical-U-pre}
    U(t)=\frac{t}{\,\mu\,}+{\mathcal{S}_\mu}-\frac{1}{\,\mu\,}\rho(t)
    +\frac{2\mathcal{S}_\mu}{\mu t}\,L_\alpha(t),
\end{equation}
    where $L_\alpha(t)\big/L(t)\to1$ as $t\to\infty$. Thus, the only remaining point is to verify
    that the last term in \eqref{eq:critical-U-pre} is asymptotically negligible relative to $\rho(t)$.
    This is precisely where the critical nature of the case $\alpha=2$ becomes visible: the dominant
    correction is governed by the integrated SV tail, whereas the term $L_\alpha(t)\big/t$
    remains only of local order.

    Setting
\[
    {\zeta_L(t)}:=\int_t^\infty \frac{L(u)}{u}\,du,
\]
    the Uniform Convergence Theorem yields
\[
    {\zeta_L(t)}-{\zeta_L(2t)}=\int_t^{2t}\frac{L(u)}{u}\,du \sim{L(t)}\int_t^{2t}\frac{du}{u}
    =\bigl(\log{2}\bigr)\,L(t) \qquad \text{as} \quad t\to\infty.
\]
    On the other hand, since ${\zeta_L\in\mathcal{SV}_\infty}$, we have
\[
    {\zeta_L(t)}-{\zeta_L(2t)}={\zeta_L(t)}\left(1-\frac{\zeta_L(2t)}{\zeta_L(t)}\right)
    =o\bigl({\zeta_L(t)}\bigr) \qquad \text{as} \quad t\to\infty.
\]
    Thus, comparing the two asymptotic relations above, we conclude that
\[
    L(t)=o\bigl({\zeta_L(t)}\bigr) \qquad \text{as} \quad t\to\infty.
\]
    At the same time, by {\hyperref[AISRTh4]{Theorem~\ref{AISRTh4}}}, $\rho(t)\sim {\zeta_L(t)\big/\mu}$,
    and therefore
\[
    \frac{L_\alpha(t)}{t\rho(t)}=o(1)  \qquad \text{as} \quad t\to\infty.
\]
    Substituting this into \eqref{eq:critical-U-pre}, we arrive at
\[
    U(t)=\frac{t}{\,\mu\,}+{\mathcal{S}_\mu}-\frac{1}{\,\mu\,}\rho(t)\bigl(1+o(1)\bigr)
    \qquad \text{as} \quad t\to\infty,
\]
    which is exactly \eqref{eq:critical-U}. This completes the proof.
\end{proof}

\begin{proof}[{Proof of {\hyperref[AISRTh5]{Theorem~\ref{AISRTh5}}}}]
    Set
\[
    v(t):=u(t)-\frac{1}{\,\mu\,}{\raise1.5pt\hbox{.}}
\]
    Thus $v(t)$ represents the deviation of the renewal density from its limiting level.
    As in the proof of {\hyperref[AISRTh3]{Theorem~\ref{AISRTh3}}}, the problem reduces
    to identifying the precise asymptotic scale of this remainder.

    Since the inter-arrival distribution $F$ admits a density $f$, the renewal density $u$ exists
    and its LT  satisfies $\mathcal{L}u(s)=\widehat{U}(s)-1$. Hence
\begin{equation}\label{eq:T5-Lv-final}
    \mathcal{L}v(s)=\mathcal{L}u(s)-\frac{1}{\mu{s}}
    =\widehat{U}(s)-1-\frac{1}{\mu{s}} {\raise1.5pt\hbox{.}}
\end{equation}
    We now invoke the following critical refinement of the renewal transform established in
    {\hyperref[AISRLem-critical-Uhat-refined]{Lemma~\ref{AISRLem-critical-Uhat-refined}}}:
\begin{equation}\label{eq:T5-Uhat-refined-final}
    \widehat{U}(s)-\frac{1}{\mu{s}}={\mathcal{S}_\mu}-\frac{1}{\,\mu\,}h(s)+o\!\bigl(h(s)\bigr)
    \qquad \text{as} \quad s\downarrow{0},
\end{equation}
    where $h(s)=\mu{\mathcal{S}_\mu}-\mathcal{L}Q_F(s)$ .

    Substituting \eqref{eq:T5-Uhat-refined-final} into \eqref{eq:T5-Lv-final}, we obtain
\[
    \mathcal{L}v(s) =({\mathcal{S}_\mu}-1)-\frac{1}{\,\mu\,}h(s)+o\!\bigl(h(s)\bigr)
    \qquad \text{as} \quad s\downarrow{0}.
\]
    Since $h(s)\to{0}$ as $s\downarrow{0}$, it follows that
\[
    \mathcal{L}v(0)=\int_0^\infty{v(t)}\,dt={\mathcal{S}_\mu}-1,
\]
    and hence
\[
    \mathcal{L}v(0)-\mathcal{L}v(s)=\frac{1}{\,\mu\,}h(s)+o\!\bigl(h(s)\bigr)
    \qquad \text{as} \quad s\downarrow{0}.
\]
    On the other hand, by {\hyperref[AISRLem-critical2]{Lemma~\ref{AISRLem-critical2}}},
    $h(s)\sim \rho(1/s)$. Consequently,
\[
    \mathcal{L}v(0)-\mathcal{L}v(s) \sim \frac{1}{\,\mu\,}\rho(1/s)
    \sim \frac{1}{\mu^2}\int_{1/s}^\infty \frac{L(u)}{u}\,du
    \qquad \text{as} \quad s\downarrow{0}.
\]
    We are therefore exactly in the setting of the borderline Tauberian transfer,
    {\hyperref[AISRLem-Tauber-critical]{Lemma~\ref{AISRLem-Tauber-critical}}}.
    Since the remainder $v(t)$ is ultimately monotone by assumption, that lemma yields
\[
    v(t)\sim \frac{1}{\mu^2}\,\frac{L(t)}{t}
    \qquad \text{as} \quad t\to\infty.
\]
    Recalling the definition $v(t)=u(t)-1/\mu$, we arrive at the explicit asymptotic representation
\[
    u(t)=\frac{1}{\,\mu\,}+\frac{1}{\mu^2\beta}\,t^{-\beta}L_\beta(t),
\]
    where $L_2(t)\big/L(t)\to1$ as $t\to\infty$. This is precisely~\eqref{AISR2.12}
    and completes the proof.
\end{proof}

    This completes the analysis of the second-order renewal asymptotics in the finite-variance RV regime,
    including the critical threshold $\alpha=2$.

\section{Concluding Remarks}   \label{IRSec5}

    The results obtained in this work provide a unified asymptotic description of renewal processes
    in the finite-variance RV regime. In the case $2<\alpha<3$, the second-order
    behaviour exhibits a clear two-scale structure: the equilibrium tail $Q_F(t)$ governs the local
    renewal correction, while the integrated tail $\rho(t)$ determines the global deviation from equilibrium.

    At the critical threshold $\alpha=2$, this structure undergoes a genuine transition. The power-law
    hierarchy collapses, and the dominant correction is no longer of fixed RV index, but instead is governed
    by an integrated SV tail. Nevertheless, the local renewal correction mechanism persists: the
    term $\Delta_R(t)$ remains asymptotically proportional to $Q_F(t)$, preserving the intrinsic
    compensation structure of the renewal equation. This reveals a structural dichotomy in second-order
    renewal asymptotics: while the global scale of fluctuations is sensitive to the tail index, the local
    renewal mechanism remains robust. In this sense, the critical regime acts as a boundary between
    qualitatively distinct asymptotic behaviours, while still fitting into a unified analytic framework.

    The structural picture developed in this paper suggests several directions for further research,
     both in refining renewal asymptotics and in their applications to stochastic population models.

\begin{itemize}

\item[(i)] \textbf{Refined asymptotics for the renewal density.}
    The present analysis focuses primarily on the renewal function $U(t)$.
    A natural extension is to establish analogous second-order asymptotics
    for the renewal density $u(t):=U'(t)$. In the regime $2<\alpha<3$, one expects
    a similar two-scale structure, while in the critical case $\alpha=2$ the
    borderline nature of the integrated tail should manifest at the level of derivatives.
    A systematic treatment would require delicate differentiation of RV asymptotics,
    possibly within de Haan theory.

\item[(ii)] \textbf{Uniform and functional versions of renewal asymptotics.}
    Another direction is to strengthen pointwise asymptotics to uniform estimates
    over expanding time intervals and to develop functional limit analogues.
    Such results would provide a more robust description of convergence towards
    equilibrium and be useful in applications involving entire trajectories.

\item[(iii)] \textbf{Higher-order and non-polynomial corrections.}
    The second-order terms identified here suggest a richer hierarchy of asymptotic
    corrections. In the finite-variance RV regime, it is natural to ask whether further
    refinement is possible, especially in the critical case $\alpha=2$, where logarithmic
    and SV effects arise. This may require extensions of classical
    Tauberian techniques to capture finer scales.

\item[(iv)] \textbf{Renewal equations and functionals of renewal processes.}
    The results can be extended to renewal-type equations of the form
\[
    H(t)=g(t)+\int_{[0,t]}H(t-x)\,dF(x),
\]
    where $g$ is a forcing function. Understanding how second-order asymptotics
    propagate through such equations would yield precise asymptotics for a broad
    class of renewal functionals, including occupation times and reward processes.

\item[(v)] \textbf{Applications to Bellman--Harris branching processes.}
    Renewal theory is central in the study of age-dependent Bellman--Harris processes.
    The refined asymptotics developed here open the possibility of sharpening classical
    limit theorems in heavy-tailed settings. In particular, one may expect:
\begin{itemize}
    \item improved asymptotics for the mean population size;
    \item refined survival probability estimates in the subcritical regime;
    \item second-order corrections for population functionals.
\end{itemize}
    These effects suggest that similar structural phenomena may persist at the level
    of branching dynamics.

\item[(vi)] \textbf{Towards a structural theory of second-order renewal phenomena.}
    The results indicate that second-order renewal asymptotics possess an intrinsic
    structure governed by the interaction between local renewal corrections and
    global tail accumulation. Further development of this viewpoint may lead to a
    unified theory encompassing both classical and heavy-tailed regimes.

\end{itemize}

\section*{Author Contributions}
    A.Imomov conceived the study, developed the theoretical framework, defined the research direction, and
    supervised the overall project. Sh.Rizaqulov conducted the literature review, performed technical verification,
    and prepared and formatted the formulas. Both authors contributed to the writing and approved the final manuscript.

\section*{Competing Interests}
    The authors declare that there are no competing interests.

\section*{Funding}
    The authors received no financial support for the authorship and publication of this paper.

\section*{Acknowledgment}

    This research work was carried out within the framework of the project program FL-8824063218 of
    the Ministry of Higher Education, Science, and Innovations of the Republic of Uzbekistan.

\end{document}